\documentclass[12pt,a4paper]{amsart}
\usepackage[utf8]{inputenc}
\usepackage[T1]{fontenc}
\usepackage{lmodern}
\usepackage[english]{babel}
\usepackage{amsmath,mathtools,amsthm,amssymb,amsfonts}
\usepackage{color}
\usepackage{geometry}
\usepackage{mathrsfs}
\usepackage{graphicx}
\usepackage{enumitem}
\usepackage{tikz} 
\usepackage{pgfplots}
\usepackage{esvect}
\usepackage{pdfpages}
\usetikzlibrary{shapes,arrows}
\usetikzlibrary{calc}
\usetikzlibrary{shapes.geometric}
\usetikzlibrary{shapes.arrows}
\usetikzlibrary{arrows.meta}
\usetikzlibrary{decorations.markings}
\usetikzlibrary{fit}
\usetikzlibrary{patterns}
\usetikzlibrary{hobby}
\usepgfplotslibrary{patchplots}
\pgfplotsset{compat=1.11}
\usepackage{hyperref}
\usepackage{nicefrac}
\usepackage{cancel}
\usepackage{parskip} 

\theoremstyle{plain}
\newtheorem{theorem}{Theorem}[section]

\newtheorem{lemma}[theorem]{Lemma}
\theoremstyle{definition}
\newtheorem{definition}[theorem]{Definition}
\newtheorem{remark}[theorem]{Remark}
\newtheorem*{theorem*}{Theorem}
\newtheorem*{remark*}{Remark}
\newtheorem*{thm*}{Theorem}
\newtheorem*{conjecture*}{Conjecture}

\newcommand{\definedas}{\mathrel{\raise.095ex\hbox{\rm :}\mkern-5.2mu=}}
\newcommand{\asdefined}{\mathrel{=\mkern-5.2mu\raise.095ex\hbox{\rm :}}}

\newcommand{\diver}{\operatorname{div}}
\newcommand{\tr}{\operatorname{tr}}

\newcommand{\Hess}{\operatorname{Hess}}
\newcommand{\diam}{\operatorname{diam}}

\newcounter{flabelcounter}
\allowdisplaybreaks

\title[]{ Local foliations by critical surfaces of the Hawking energy and small sphere limit}
\author[Pe\~nuela Diaz]{Alejandro Pe\~nuela Diaz}
\address{Max Planck Institute for Gravitational Physics, University of Potsdam, 14476 Potsdam, Germany}
\email{alejandro.penuela@aei.mpg.de}

\begin{document}
\begin{abstract}
Local foliations of area constrained Willmore surfaces on a $3$-dimensional Riemannian manifold were constructed by  Lamm, Metzger and Schulze  \cite{locwill}, and   Ikoma, Machiodi and Mondino \cite{ikoma}, the leaves of these foliations are in particular critical surfaces of the Hawking energy in case they are contained in a totally geodesic spacelike hypersurface. We generalize these foliations to the general case of a non-totally geodesic spacelike hypersurface, constructing an  unique  local foliation of area constrained critical surfaces of the Hawking energy.   A discrepancy when evaluating the so called small sphere limit of the Hawking energy was found by Friedrich  \cite{Alex}, he studied concentrations of  area constrained critical surfaces of the Hawking energy and obtained  a result that apparently differs from  the well established small sphere limit of the Hawking energy of Horowitz and Schmidt  \cite{Haw}, this small sphere limit  in principle must be satisfied by any quasi local energy.  We independently  confirm  the  discrepancy and explain the reasons for it to happen. We also prove that these surfaces are suitable to evaluate the Hawking energy in the sense of Lamm, Metzger and Schulze \cite{willflat}, and we find an indication that these surfaces may induce an excess in the energy measured.
\end{abstract}
\maketitle
\section{Introduction and Results}
The search for a quasi local energy is one of the most prominent problems in classical relativity, with many different candidates (for a detailed review of the topic see \cite{Living}). From these candidates one of the most famous is the quasi local energy described by Hawking in 1968 \cite{Hawma}, the so called Hawking energy,   given by the expression 
\begin{equation}\label{Hawkingene}
     \mathcal{E}(\Sigma) = \sqrt{\frac{|\Sigma|}{16 \pi}} \left( 1+ \frac{1}{8 \pi} \int_\Sigma \theta^+ \theta^- d\mu \right),
\end{equation}
where $\Sigma$ is a closed surface in a $4$ dimensional space time, $|\Sigma|$ is the area of the surface, and $\theta^+ \theta^-$ is the product of the null expansions $\theta^+$ and $\theta^-$.  The Hawking energy is one of the simplest quasi local energies that one can find and fulfils almost all the expected properties of a quasi local energy, however it has the inconvenience that it is not necessarily positive, there are well known examples in flat space of surfaces that give a negative Hawking energy (Hayward defined a generalization of the Hawking energy in \cite{Hayward} to address this problem. Nevertheless, we will consider Hawking's definition). Therefore it is of high importance to know which surfaces are appropriate to evaluate the Hawking energy, for instance, it  was shown by Christodoulou and Yau in \cite{Chriyau} and by Miao, Wang and Xie in \cite{Miao} that under some physically reasonable conditions the Hawking energy (in the time symmetric case) is well behaved when evaluated in constant mean curvature spheres. 

This paper is divided into two parts, one devoted to studying foliations of area constrained critical surfaces of the Hawking energy, and other devoted to studying an apparent  discrepancy of the small sphere limit when approaching a point in spacelike direction.

\subsection{Foliations}  We will work in the initial data set setting, this means that we consider a smooth $3$-dimensional Riemannian manifold $(M,g)$, which will be equipped with a symmetric $2$-tensor $k$, we denote this manifold as a triple $(M,g,k)$.  The motivation for considering this setting comes again from general relativity since  $(M,g,k)$  can be seen as a spacelike hypersurface  with second fundamental form $k$ in a $4$-dimensional spacetime. In this setting the Hawking energy can be written for a surface $\Sigma \subset M $ as
\begin{equation}\label{hawkingmass2}
    \mathcal{E}(\Sigma) = \sqrt{\frac{|\Sigma|}{16 \pi}} \left( 1- \frac{1}{16 \pi} \int_\Sigma H^2 - P^2 d\mu \right),
\end{equation}
where $H$ is the mean curvature of the surface $\Sigma$ and   $P=\tr_{g_{\Sigma}} k$ is the trace of the tensor $k$ with respect to the metric induced in $\Sigma$, that is  $P= \tr_\Sigma k= \tr k -k(\nu,\nu)$, where $\nu$ is the outward  normal to $\Sigma$ in $M$.

From a variational point of view studying (\ref{hawkingmass2}) is equivalent to studying the Hawking functional 
\begin{equation}\label{hawfun}
    \mathcal{H}(\Sigma)= \frac{1}{4} \int_\Sigma H^2 - P^2 d\mu
\end{equation}
We are interested in studying area constrained critical surfaces of this functional, then considering a fixed area, we look for surfaces that maximize or minimize the functional.  In particular, these are then critical surfaces of the Hawking energy. In case $k=0$, the so called time symmetric case (or a totally geodesic hypersurface) the Hawking functional reduces to the Willmore functional 
\begin{equation}\label{willfunc}
    \mathcal{W}(\Sigma)= \frac{1}{4} \int_\Sigma H^2  d\mu 
\end{equation}
and the critical surfaces of this functional subject to the constraint that $|\Sigma| $ be fixed are the area constrained Willmore surfaces which we call here for simplicity just Willmore surfaces. These surfaces are characterized by the following Euler Lagrange equation with Lagrange parameter $\lambda$.
\begin{equation}\label{Willeq}
    0= \lambda H  +\Delta^\Sigma H + H|\mathring{B}|^2+ H \mathrm{Ric}(\nu, \nu), 
\end{equation}
where $\mathring{B}$   is the traceless part of the second
fundamental form $B$ of $\Sigma$ in $M$, that is $ \mathring{B} = B- \frac{1}{2} H g_\Sigma$ with norm  $|\mathring{B}|^2 = \mathring{B}_{ij}\, g_\Sigma^{ip}\, g_\Sigma^{jq}\, \mathring{B}_{pq}$,  $\mathrm{Ric}$ is the Ricci curvature of $M$, $\nu$ is the outward normal to $\Sigma$  and $\Delta^\Sigma $ is the Laplace-Beltrami operator on $\Sigma$.

The Willmore surfaces have been extensively studied and in the context of general relativity they were first introduced by Lamm, Metzger and Schulze in \cite{willflat}, where they showed that there exist  a unique foliation of Willmore spheres for asymptotically flat manifolds, this is a foliation that covers the whole manifold except a compact region, what we call a foliation at infinity. In their work they claimed that these surfaces are the optimal surfaces for evaluating the Hawking energy, this since if the manifold has nonnegative  scalar curvature (that means that the dominant energy condition holds) the Hawking energy is nonnegative on these surfaces and it is monotonically nondecreasing along the foliation. It was also shown in \cite{Thomas} by Koerber that the leaves of the foliation are strict local area preserving maximizers of the Hawking energy.

This foliation by Willmore spheres at infinity has been improved by Eichmair and Koerber in \cite{eichko}  where they used a Lyapunov-Schmidt reduction procedure (a technique that  will be also applied in our construction) to obtain the foliation, furthermore, in \cite{willcen} they studied the center of mass of this foliation. The non-totally geodesic case was also considered by Fridrich in his thesis \cite{Friedrich2020}, where he generalized the foliation of \cite{willflat} for critical surfaces of the Hawking functional and showed that the Hawking energy is monotonically nondecreasing along the foliation. We will see in Theorem \ref{positivity} that under even more general conditions, if the dominant energy condition holds then,  the Hawking energy is nonnegative on these surfaces for a large enough radius.
\begin{theorem*}
Assuming that on an asymptotically flat initial data set $(M,g,k)$ the dominant energy conditions holds. There exist an $r_0>0$ such that for $r\geq r_0$,  if $\Sigma_r$ is a critical surface of the Hawking energy  with area radius $r$ ( $|\Sigma_r |=4 \pi r^2$), it is almost  centered, the Lagrange parameter $ \lambda$ is positive with  $ \lambda = \mathcal{O}(r^{-3})$  and also the mean curvature is positive with $H = \mathcal{O}(r^{-1}) $  then  the Hawking energy on $\Sigma_r$ is nonnegative.
\end{theorem*}
This shows that  the Hawking functional critical surfaces  in the asymptotically flat case have the same desirable properties as the Willmore surfaces and are "optimal" (in the sense of Lamm, Metzger and Schulze) to evaluate the Hawking energy on a spacelike hypersurface.

Here we are more interested in the local behaviour of the surfaces; in this direction, it was shown by Lamm and Metzger in \cite{ToMet} and later by Laurain and Mondino in \cite{Laurain} that Willmore surfaces concentrated around points which are critical points of the scalar curvature, that is points $p\in M$ such that  $\nabla \mathrm{Sc}_p =0$. Furthermore in \cite{locwill}  Lamm, Metzger and Schulze, and in \cite{ikoma} Ikoma, Machiodi and Mondino showed by a means of a Lyapunov-Schmidt reduction procedure that if at a point $p \in M$, $\nabla \mathrm{Sc}_p =0$ and  $\nabla^2 \mathrm{Sc}_p$ is not degenerated then around $p$ there is a local foliation of area constrained Willmore surfaces around that point. 

The first part of this paper will be devoted to generalizing these local foliations to the general case when $k \neq 0$, obtaining the following results.  
\begin{theorem*}
Let $p \in M $ be such that at $p$, $\nabla(  \mathrm{Sc} + \frac{3}{5} (\tr k)^2 + \frac{1}{5} |k|^2)=0  $ and $\nabla^2(  \mathrm{Sc} + \frac{3}{5} (\tr k)^2 + \frac{1}{5} |k|^2) $ is nondegenerate. Then there exist  $\delta, \epsilon_0, C >0$ such that if at $p$, 
$$ C |(\nabla^2(  \mathrm{Sc} + \frac{3}{5} (\tr k)^2 + \frac{1}{5} |k|^2))^{-1}| \cdot \; |k|\; |\nabla k|\, (|k|^2 + |\mathrm{Ric}|  ) <1 $$ then there exist a smooth foliation $\mathcal{F}= \{ S_r : r\in (0, \delta)  \} $  around $p$ of  area constrained critical spheres of the Hawking functional, that is surfaces satisfying equation (\ref{eulag}), for some $\lambda \in \mathbb{R}$. Furthermore these surfaces can be express  as normal graphs over geodesic spheres of radius $r$, and they  satisfy $\mathcal{H}(S_r) < 4\pi +\epsilon_0^2$ and $|S_r|< \epsilon_0^2 $,  for $r \in (0, \delta)$.
\end{theorem*}
 We also obtained a uniqueness result. 
\begin{theorem*}
$(i)$ Assume that at $p$, $\nabla(  \mathrm{Sc} + \frac{3}{5} (\tr k)^2 + \frac{1}{5} |k|^2)=0  $, $\nabla^2(  \mathrm{Sc} + \frac{3}{5} (\tr k)^2 + \frac{1}{5} |k|^2) $ is nondegenerate and that the foliation $ \mathcal{F}$ of the previous theorem  exists satisfying $\mathcal{H}(\Sigma) < 4\pi +\epsilon_0^2$ and $|\Sigma|< \epsilon_0^2 $ for any $\Sigma \in \mathcal{F} $ and the $\epsilon_0$ of the theorem. If  $ \mathcal{F}_2$ is a  foliation around $p$ of  area constrained critical spheres of the Hawking functional, which satisfy $\mathcal{H}(\Sigma) < 4\pi +\epsilon^2$ and $|\Sigma|< \epsilon^2 $ for any $\Sigma \in \mathcal{F}_2 $ and some $\epsilon \leq \epsilon_0$, then either $\mathcal{F}$ is a restriction of $\mathcal{F}_2$ or $\mathcal{F}_2$ is a restriction of $\mathcal{F}$.

$(ii)$ Claim $(i)$  also holds, if instead of  foliations, we consider a concentration of  surfaces around $p$ that satisfy $\mathcal{H}(\Sigma) < 4\pi +\epsilon^2$ and $|\Sigma|< \epsilon^2 $ for any $\Sigma \in \mathcal{F}_2 $ and  $\epsilon \leq \epsilon_0$.
\end{theorem*}
\subsection{Small Sphere Limit}\label{secsmall}

For the second part of this paper, we will focus on studying the small sphere limit of the Hawking energy. In general, any quasi local energy must have the right asymptotics when evaluated on large and small spheres. In particular it must satisfy the small sphere limit. 

Here we consider a $4$-dimensional spacetime $M^4$ and will denote the geometric quantities on this manifold by an index $(\cdot)^4$. Before introducing the small sphere limit we need to define what a light cut is. 

Let $p\in M^4$ and let $C_p$ be the future null cone of $p$, that is, the null hypersurface generated by future null geodesics starting at $p$. Pick any future directed timelike unit vector $e_0$ at $p$.  We normalize a null vector $L$ at $p$ by $\langle L,e_0 \rangle =-1$. We consider the null geodesics of the vector  $L$ and let $l$ be the affine parameter of these null geodesics. We define  the light cuts $\Sigma_l$ to be the family of surfaces on $C_p$ determined by the level sets of the
affine parameter $l$.

Let $p\in M^4$ and let $C_p$ be the future null cone of $p$, that is, the null hypersurface generated by future null geodesics starting at $p$. Pick any future directed timelike unit vector $e_0$ at $p$.  We normalize a null vector $L$ at $p$ by $\langle L,e_0 \rangle =-1$. We consider the null geodesics of the vector  $L$ and let $l$ be the affine parameter of these null geodesics. We define  the light cuts $\Sigma_l$ to be the family of surfaces on $C_p$ determined by the level sets of the
affine parameter $l$.

 The small sphere limit tells us that when evaluating the quasi local energy on surfaces approaching a point $p$, in a spacetime along the light cuts of the null cone of $p$, the leading term of the quasi local energy should recover the stress energy tensor in spacetimes with matter fields, i.e.,  $\lim_{r \to 0} \frac{M(\Sigma_r)}{r^3} = \frac{4 \pi}{3} T(e_0,e_0)$.  If the point is contained in a spacelike hypersurface $M \subset M^4$ then by using the Gauss–Codazzi
equations we obtain
$$\lim_{r \to 0} \frac{M(\Sigma_r)}{r^3} = \frac{4 \pi}{3} T(e_0,e_0) =\frac{1}{12}( \mathrm{Sc} + (\tr k)^2 - |k|^2 ), $$
where everything is evaluated at $p$, and the right hand side is the energy density of the Einstein constrained equations on  $M$ (here $\mathrm{Sc} $ and $k$ are the scalar curvature and second fundamental form of $M$). The small sphere limit was first introduced by Horowitz and Schmidt  for the Hawking energy \cite{Haw}, it must be satisfy by any reasonable notion of quasi local energy as it was shown for the Brown-York energy \cite{Brown} the Kijowski-Epp-Liu-Yau  energy \cite{Yu}, the Wang-Yau \cite{wangyau} and for their higher dimensional  versions \cite{Wang} among others. In particular, when the point $p$ is contained in a spacelike hypersurface $M \subset M^4$, we have the following expansion for the Hawking energy for cuts on the light cut $S_l$
\begin{equation}\label{lightex}
   \mathcal{E}(\Sigma_l)=   \frac{1}{12}( \mathrm{Sc} + (\tr k)^2 - |k|^2 )l^3 + \mathcal{O}(l^5)
\end{equation}
at $p$. Having this expansion in mind when studying area constrained critical surfaces of the Hawking functional (\ref{hawfun}) in a spacelike hypersurface (initial data set), it would be natural to think that such surfaces concentrate around points satisfying that
\begin{equation}
    \nabla (\mathrm{Sc} + (\tr k)^2 - |k|^2)=0
\end{equation}
 at $p$. However, in \cite{Alex} Friedrich found that this is not the case. In fact a point having a concentration of these surfaces must satisfy 
$$\nabla(  \mathrm{Sc} + \frac{3}{5} (\tr k)^2 + \frac{1}{5} |k|^2)=0  $$
at $p$, this was an unexpected result that we managed to confirm with our results as well (in Theorem \ref{primfoli}) and we also obtained in the equivalent Theorem \ref{nonexis}. This result gives the  impression that the local expansion of the Hawking energy depends on how you approach the point.  Figure \ref{figure} illustrates the situation.

In section \ref{secdis}, we will  study this discrepancy found by Friedrich and see that it comes from purely geometric reasons, in particular, that even if  a priori the two ways to approach the point may look similar, the surfaces used  are quite different. Finally, in Remark \ref{excess} we will see that these results suggest that  the critical surfaces of the Hawking functional induce an excess in the measure of  the Hawking energy. 
\begin{figure}[h]\label{figure}
\centering
\includegraphics[width=8.7cm]{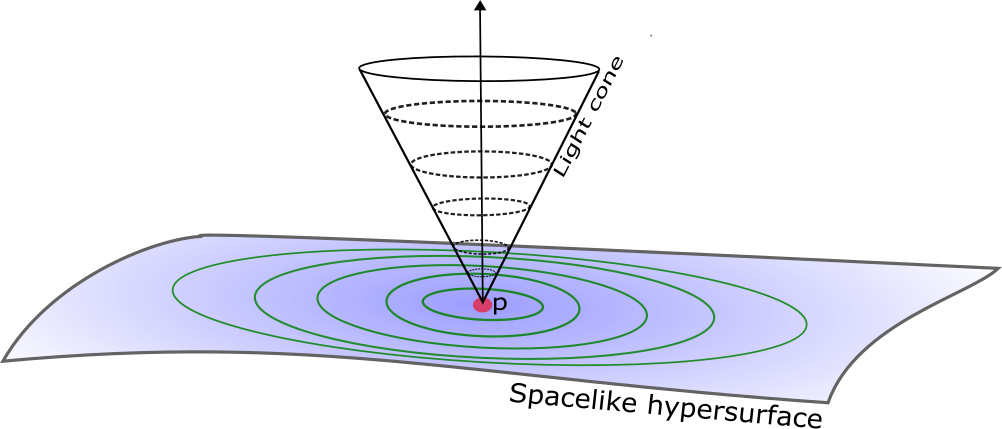}
\caption{ Comparison between approaching a point along cuts on a null cone and along critical surfaces on a spacelike hypersurface.}
\end{figure}

\section{Foliations}\label{sec2}
\subsection{Preliminaries and setting}
In this section, we work with data $(M,g,k)$ where $(M,g)$ is a smooth $3$-dimensional Riemannian manifold which is equipped with a symmetric $2$-tensor $k$. In General relativity, the data $(M,g,k)$ represents a spacelike hypersurface (or an inital data set) with second fundamental form $k$ in a $4$-dimensional spacetime. In this setting we don't need any mention for the spacetime.
We introduce the following notation: The covariant derivatives will be denoted by $;$ and the partial derivatives $\frac{\partial}{\partial x^i}$  by a comma  or by $\partial_i$.  

Now we derive the equation that characterizes the area surfaces equations of the Hawking functional.
\begin{lemma}[First variation]
The area constrained Euler Lagrange equation for the Hawking functional (\ref{hawfun}) is 
\begin{equation}\label{eulag}
\begin{split}
   0=& \lambda H  +\Delta^\Sigma H + H|\mathring{B}|^2+ H \mathrm{Ric}(\nu, \nu)+P( \nabla_\nu \tr k - \nabla_\nu k(\nu,\nu )) - 2P \diver_\Sigma (k(\cdot, \nu))\\ & +\frac{1}{2}H P^2 - 2k (\nabla^\Sigma P, \nu )
    \end{split}
\end{equation}
Here $H$ is the mean curvature of $\Sigma$ , $\mathring{B}$   is the traceless part of the second
fundamental form $B$ of $\Sigma$ in $M$, that is
$\mathring{B}= B- \frac{1}{2} H g_\Sigma$ where $g_\Sigma$ is the induced
metric on $\Sigma$,  $\mathrm{Ric} $ is the Ricci curvature of $M$, $\nabla^\Sigma $, $\diver_\Sigma$ and $\Delta^\Sigma$ are the covariant derivate, tangential divergence  and Laplace Beltrami operator on $\Sigma$. Finally $\lambda \in \mathbb{R}$ plays the role of a Lagrange parameter.
\end{lemma}

\begin{proof}
Let $\Sigma \subset M$ be a surface and let $f: \Sigma \times (-\epsilon , \epsilon) \rightarrow M$ be a variation of $\Sigma$ with $f(\Sigma, s)= \Sigma_s$ and lapse $\frac{\partial f}{\partial s }_{|s=0}=\alpha \nu $. In \cite[Section 3]{willflat}, it was shown that the first variation of the Willmore functional (\ref{willfunc}) is given by 
\begin{equation}
\begin{split}
  \frac{1}{4}  \frac{d}{d s}\int_{\Sigma_s} H^2  d\mu_{| s=0}  = \int_{\Sigma_s} \left( -\Delta^\Sigma  H - H|\mathring{B}|^2- H\mathrm{Ric}(\nu, \nu)  \right) \alpha \, d\mu,
    \end{split}
\end{equation}
now let's compute the variation of $\frac{1}{2} \int_\Sigma P^2  d\mu $. In \cite{Ce}, it was shown that the variation of $P$ is given by
\begin{equation}
    \frac{d\, P }{d s}_{| s=0}=\left( \nabla_\nu \tr k - \nabla_\nu k(\nu, \nu)\right)\alpha +2 k(\nabla \alpha, \nu),
\end{equation}
using this relation and integration by parts we have 
\begin{equation}
\begin{split}
 \frac{1}{4}  \frac{d}{d s}\int_{\Sigma_s} P^2  d\mu_{| s=0}  =& \int_{\Sigma_s} \frac{1}{2} P^2 H  \alpha +P\left( \nabla_\nu \tr k - \nabla_\nu k(\nu, \nu)\right)\alpha +2P k(\nabla \alpha, \nu)  d\mu\\
 =\int_{\Sigma_s}\big( &\frac{1}{2} P^2 H   +P\left( \nabla_\nu \tr k - \nabla_\nu k(\nu, \nu)\right) -2P \diver_\Sigma \left( k(\cdot, \nu) \right )\\
 &- 2 k(\nabla^\Sigma P, \nu ) \big) \alpha  d\mu.
    \end{split}
\end{equation}
We are considering area constrained surfaces, which means surfaces whose variation of area is zero. This traduces to the area constraint $\int_\Sigma H \alpha d\mu =0 $. Then our surfaces must satisfy the area constraint and 
\begin{equation*}
\begin{split}
 &0=\frac{1}{2}  \left( \frac{d}{d s}\int_{\Sigma_s} H^2  d\mu_{|s=0} -  \frac{d}{d s}\int_{\Sigma_s} P^2  d\mu_{|s=0} \right)  =\\
 &\int_{\Sigma_s} \big( -\Delta^\Sigma H - H|\mathring{B}|^2- H \mathrm{Ric}(\nu, \nu) - \frac{1}{2} P^2 H   -P\left( \nabla_\nu \tr k - \nabla_\nu k(\nu, \nu)\right) +2P \diver_\Sigma \left( k(\cdot, \nu) \right )\\
 &+ 2 k(\nabla^\Sigma P, \nu ) \big) \alpha  d\mu
    \end{split}
\end{equation*}
 Then combining  this expression and the area constraint give us the Euler Lagrange equation (\ref{eulag}).
\end{proof}
Note that this result is equivalent to \cite[Lemma 2.8]{Alex}, and it reduces to the Willmore equation (\ref{Willeq}) in case $k=0$.

 Friedrich proved in \cite{Friedrich2020}  the existence of a foliation of critical surfaces of the Hawking functional in asymptotically Schwarschild manifolds, and also proved that the Hawking energy is monotonically nondecreasing along the foliation. Now we will show that if the dominant energy condition holds, the Hawking energy is nonnegative on these surfaces. This holds in more general conditions that the ones considered by Friedrich (it holds when assuming general  asymptotic flatness). First, recall that the dominant energy condition is given by 
\begin{equation}
    \mu \geq |J|
\end{equation}
where 
\begin{equation}
     \mathrm{Sc} + (\tr k)^2 - |k|^2=2 \mu \quad \text{and} \quad  \diver (k -(\tr k) g)= J
\end{equation}
are the energy density and the momentum density of the Einstein constraint equations. In particular, the dominant energy condition implies $\mu \geq 0$ which also implies $ \mathrm{Sc} + \frac{2}{3}(\tr k)^2 \geq 0 $.

\begin{theorem}\label{positivity}
Assuming that on an asymptotically flat initial data set $(M,g,k)$, where $k$ decays like $ |k|+ |\nabla k| |x| \leq C|x|^{-\frac{3}{2}-\epsilon}$ for some constant $C>0$ and $\epsilon \in (0, \frac{1}{2})$ and  the dominant energy conditions holds. There exist an $r_0>0$ such that for $r\geq r_0$,  if $\Sigma_r$ is a critical surface of the Hawking energy  with area radius $r$ ( $|\Sigma_r |=4 \pi r^2$), it is almost  centered ($|x|$ the distance to the origin of any point in $\Sigma_r$ is comparable to $r$), the Lagrange parameter $ \lambda$ is positive with  $ \lambda = \mathcal{O}(r^{-3})$  and also the mean curvature is positive with $H = \mathcal{O}(r^{-1}) $  then  the Hawking energy on $\Sigma_r$ is nonnegative.
\end{theorem}

\begin{proof}
According to (\ref{hawkingmass2}), it is enough to see that $\int_{\Sigma_r} H^2 -P^2 d\mu \leq 16 \pi$. We proceed similarly as in \cite[Theorem 4]{willflat}. We consider  equation (\ref{eulag}), divided by $H$, integrate by parts the term $\frac{\Delta H}{H}$ and use the Gauss equation $2\mathrm{Ric}(\nu, \nu) = \mathrm{Sc} -\mathrm{Sc}^{\Sigma_r} + H^2 -|B|^2  $ obtaining
\begin{equation*}
\begin{split}
   0=  \int_{\Sigma_r}&\lambda+ |\nabla \log H|^2 + \frac{1}{2} |\mathring{B}|^2+  \frac{1}{2}(\mathrm{Sc} -\mathrm{Sc}^{\Sigma_r}) +\frac{P}{H}( \nabla_\nu \tr k - \nabla_\nu k(\nu,\nu ))\\& +\frac{1}{4}H^2 + \frac{1}{2} P^2 - 2\frac{P}{H} \diver_\Sigma (k(\cdot, \nu))- \frac{2}{H} k (\nabla^\Sigma P, \nu )d\mu. 
    \end{split}
\end{equation*}
  We can estimate for some constant $C$
\begin{equation*}
    \int_{\Sigma_r}\lambda+ |\nabla \log H|^2 + \frac{1}{2} |\mathring{B}|^2   +\frac{1}{4}H^2  +\frac{1}{2} P^2  - \frac{C}{H}|k| |\nabla k| d\mu \leq  - \int_{\Sigma_r} \frac{1}{2}(\mathrm{Sc} -\mathrm{Sc}^{\Sigma_r})  d\mu.
\end{equation*}
Now using  Gauss-Bonnet theorem to replace $ \mathrm{Sc}^{\Sigma_r}$  and subtracting $\frac{1}{3}(\tr k)^2$ on both sides we have 
\begin{equation*}
\begin{split}
    &\int_{\Sigma_r}\lambda+ |\nabla \log H|^2 +\frac{1}{4}(H^2- P^2)  +\frac{3}{4} P^2 - \frac{1}{3}(\tr k)^2+ \frac{1}{2} |\mathring{B}|^2   - \frac{C}{H}|k| |\nabla k| d\mu \\& \leq 4 \pi - \int_{\Sigma_r} \frac{1}{2}(\mathrm{Sc}+\frac{2}{3} (\tr k)^2) d\mu.
    \end{split}
\end{equation*}
Now  thanks to the dominant energy condition, we have $\mathrm{Sc}-\frac{2}{3} (\tr k)^2 \geq 0  $ and by the decay conditions of the assumptions, it is direct to see that for $r$ large enough
$$0 \leq\int_{\Sigma_r}\lambda+ \frac{3}{4} P^2 - \frac{1}{3}(\tr k)^2  - \frac{C}{H}|k| |\nabla k| d\mu,  $$
then it follows directly that  $\int_{\Sigma_r} H^2 -P^2 d\mu \leq 16 \pi$. 
\end{proof}

\begin{remark}
 Note that the foliation constructed in \cite{Friedrich2020} satisfies the conditions of the previous result. This shows that these  surfaces have the same desired properties as the Willmore surfaces in the totally geodesic case ($k=0$) when evaluating the Hawking energy.
\end{remark}

    To produce our foliations, we will use the fact that geodesics spheres of small radius around a point $p\in M$ form a foliation, and this foliation can be perturbed in a suitable way. The perturbation procedure consists of a normal perturbation to the geodesics spheres and a perturbation of their center. For this procedure, we will consider the  setup  considered in \cite{Me},  which is like the one considered  in  \cite{ikoma,locwill,Ye} when $k=0$.

Denote by $R_p$ the injectivity radius of $p$  and define $r_p:= \frac{1}{8} R_p$. we will also denote  $\mathbb{B}_r:=\{x\in\mathbb{R}^{3}: ||x ||<r  \} $ and  $\mathbb{S}^2_r:=\{x\in\mathbb{R}^{n+1}: ||x ||=r  \} $ where $||\cdot ||$ is the euclidean norm.  

For $\tau \in \mathbb{R}^{3} $ with $|| \tau|| < r_p$ we define $F_\tau : \mathbb{B}_{2r_p} \rightarrow M $ by
\begin{equation}\label{coordinates}
    F_\tau(x)= \exp_{c(\tau)}(x^i e^\tau_i), 
\end{equation}
where $c(\tau)= \exp_p (\tau^i e_i)$, $e_i$ are an orthonormal basis of $T_p M$ and $e_i^\tau$ their parallel transport to $c(\tau)$ along the geodesic $c(t \tau )_{0\leq t \leq 1
}$.  Consider also the dilation $\alpha_r(x)=rx $ for $r>0$. For each $\tau$ and $0<r<r_p$, the map $F_\tau \circ \alpha_r $ gives rise to some rescaled normal coordinates centered at $c(\tau)$, in particular, the metric $g$ in these coordinates satisfies that 
$$g_{ij}(r x)= r^2 (\delta_{ij} + \sigma_{ij}(x r)) $$
where $\delta$ detones the euclidean metric and  $\sigma$ satisfies $ |\sigma_{ij}(x)|\leq |x|^2$, we denote this by $g_{ij}(rx)= r^2 (\delta_{ij} + \mathcal{O}(|x|^2 r^2))$.

As in \cite{locwill}, let $\Omega_1= \{ \varphi \in \mathcal{C}^{4, \frac{1}{2}}(\mathbb{S}^2) \; |\; ||\varphi ||_{\mathcal{C}^{4, \frac{1}{2}}(\mathbb{S}^2)} <\delta_0 \}$   with $\delta_0>0$ so small that $ S_\varphi :=\{x+ \varphi(x) \nu(x): x \in \mathbb{S}^2 \}$  is an embedded $\mathcal{C}^4$ surface in  $\mathbb{R}^{3} $, and    where $\nu$ is the unit normal to $\mathbb{S}^n$.  Define the map $\Tilde{\Phi}: (0,r_p) \times \mathbb{B}_{2r_p} \times \Omega_1 \times \mathbb{R} \rightarrow \mathcal{C}^\frac{1}{2} (\mathbb{S}^2)  $ given by 
\begin{equation}
\begin{split}
    \Tilde{\Phi}(r, \tau , \varphi , \lambda)=&\lambda H  +\Delta^\Sigma H + H|\mathring{B}|^2+ H \mathrm{Ric}(\nu, \nu)+\frac{1}{2}H P^2+P( \nabla_\nu \tr k - \nabla_\nu k(\nu,\nu ))  \\&- 2P \diver_\Sigma (k(\cdot, \nu))- 2k (\nabla^\Sigma P, \nu ),
    \end{split}
\end{equation}
where the expression of the right is evaluated for  $\Sigma= F_\tau (\alpha_r (S_\varphi))$ at $F_\tau (r(x+ \varphi(x) \nu))$ with respect to $g$. Note that this is the equation that characterizes the area constrained critical surfaces of  the Hawking functional. To find a foliation,  we look for  some functions $\tau(r)$, $\varphi(r)$ and $\lambda(r)$ such that $ \Tilde{\Phi}(r, \tau(r) , \varphi(r) , \lambda(r))=0 $ for some $r \in (0, r_0)$, then our surfaces  $\Sigma_r= F_{\tau(r)} (\alpha_r (S_\varphi(r)))$ are parameterized by   $r$ and with some extra work one can see that they form a foliation.

In order to find these functions, we will use the implicit function theorem, but in an auxiliary manifold $(\mathbb{B}_{2r_p}, g_{\tau,r}= r^{-2} \alpha_r^*(F_\tau^*(g)),$ $k_{\tau, r}= r^{-1} \alpha_r^*(F_\tau^*(k)) )$  this manifold  is useful since  its metric is conformal to $g$ in the $F_\tau \circ \alpha_r $ coordinates and when $r=0$,  $g_{\tau,0}$  is just the euclidean metric  and $k_{\tau,0}=0$, allowing us to work with an $r$ arbitrarily small. Furthermore, we define the operator 
\begin{equation}\label{resc}
\begin{split}
    \Phi(r, \tau , \varphi , \lambda)=&r^2 \lambda H_{r,\tau}  +\Delta_{r,\tau}^\Sigma H_{r,\tau} + H_{r,\tau}|\mathring{B}_{r,\tau}|^2+ H_{r,\tau} \mathrm{Ric}_{r,\tau}(\nu_{r,\tau}, \nu_{r,\tau})+\frac{1}{2}H_{r,\tau} P_{r,\tau}^2\\&+P_{r,\tau}( \nabla_{\nu_{r,\tau}} \tr k_{r,\tau} - \nabla_{\nu_{r,\tau}} k_{r,\tau}(\nu_{r,\tau},\nu_{r,\tau} ))  - 2P_{r,\tau} \diver_\Sigma (k_{r,\tau}(\cdot, \nu_{r,\tau}))\\&- 2k_{r,\tau} (\nabla^\Sigma P_{r,\tau}, \nu_{r,\tau} )
    \end{split}
\end{equation}
where the right hand side is evaluated on $\Sigma= S_\varphi $ at $x+ \varphi(x) \nu(x) $ with respect to $g_{\tau,r}$ on $\mathbb{B}_2 $ (we denote this by the subindex $r,\tau$). The convenience of this operator on the auxiliary manifold is that the metric $g_{\tau,r}$ is  conformal to $g$ in the coordinates $F_\tau \circ \alpha_r $ with conformal factor $r^2$, $k_{r,\tau}$ is also conformal to $k$ and then using how the different terms on  (\ref{resc}) transform under this conformal transformation (for instance, $H_{r,\tau}= r H$, $\nu_{r,\tau}= r \nu $, $P_{r,\tau}= r P $ etc) one obtains the following relation
\begin{equation}
    \begin{split}
       \Phi(r, \tau , \varphi , \lambda)=&r^3 \Tilde{\Phi}(r, \tau , \varphi , \lambda)
    \end{split}
\end{equation}
 and therefore, if we manage to find a surface satisfying $\Phi(r, \tau , \varphi , \lambda)=0$ we then have an area constrained critical surfaces of the Hawking functional in our original manifold.

Note that the operator (\ref{resc}) can be decomposed into two parts,  one that doesn't depend on $k$ that we denote by $W_1$,  and another that depends on $k$ which we denote by  $W_2 $. Then we have $\Phi(r, \tau , \varphi , \lambda)=(W_1+ W_2)(r, \tau , \varphi, \lambda) $ where
\begin{equation}
    W_1(r, \tau , \varphi, \lambda):= r^2 \lambda H_{r,\tau}  +\Delta_{r,\tau}^\Sigma H_{r,\tau} + H_{r,\tau}|\mathring{B}_{r,\tau}|^2+ H_{r,\tau} \mathrm{Ric}_{r,\tau}(\nu_{r,\tau}, \nu_{r,\tau})
\end{equation}
and
\begin{equation}
\begin{split}
    W_2(r, \tau , \varphi, \lambda):=& \frac{1}{2}H_{r,\tau} P_{r,\tau}^2+P_{r,\tau}( \nabla_{\nu_{r,\tau}} \tr k_{r,\tau} - \nabla_{\nu_{r,\tau}} k_{r,\tau}(\nu_{r,\tau},\nu_{r,\tau} ))  \\ &- 2P_{r,\tau} \diver_\Sigma (k_{r,\tau}(\cdot, \nu_{r,\tau}))  - 2k_{r,\tau} (\nabla^\Sigma P_{r,\tau}, \nu_{r,\tau} )
    \end{split}
\end{equation}
Note that $ W_1(r, \tau , \varphi, \lambda) $ corresponds to the Willmore operator whose  local behaviour has been studied in many different papers like in \cite{ToMet}, \cite{locwill} and \cite{ikoma} among others.

From now on, we will denote by $A^\tau(x)$  a tensor evaluated at $F_\tau(x)$ and then  $A^\tau$(0) is the tensor  evaluated at the point $c(\tau)$. Also  if $\tau=0$, we omit the superscript i.e., $A^0=A$.

Now let's see the operator (\ref{resc}) when one considers a geodesic sphere, that is, when $\varphi$ is equal to zero.
\begin{lemma}
Considering the setting of above one has 
 \begin{equation}
     W_1(r, \tau , 0, \lambda)=  r^2 (2 \lambda - \frac{2}{3}\mathrm{Rs}^\tau(0) + 4 \mathrm{Ric}^\tau_{pq}(0) x^p x^q ) + r^3( 5  \mathrm{Ric}^\tau_{pq,s}(0) x^p x^q x^s  -\mathrm{Rs}_{,p}  x^p )+ \mathcal{O}(r^4). 
 \end{equation}
\begin{equation}\label{sphep}
\begin{split}
    W_2(r, \tau , 0, \lambda)=& r^2\big( -(\tr k^\tau)^2 + (2\tr k^\tau \, k^\tau_{ij}
    + 4 k^\tau_{si} \, k^\tau_{sj})x^i x^j  -5  k^\tau_{ij}\, k^\tau_{pq} x^ix^j x^px^q  \big)\\
    &+ r^3 \big( \big(\frac{\partial_i (\tr k^\tau)^2}{2}  - 2\partial_s ( \tr k^\tau \, k^\tau_{ s i}) \big) x^i   +(\partial_s (\tr k^\tau k^\tau_{ij}) + 2 \partial_t (k^\tau_{ij} k^\tau_{ts}) )x^i x^j x^s \\
    &-3 k^\tau_{ij} \, k^\tau_{pq,s} x^i x^j x^px^q x^s \big) +\mathcal{O}(r^4).
    \end{split}
\end{equation}
Where $k^\tau=k^\tau(rx)$.  In particular, $\Phi(r, \tau , 0, \lambda)=(W_1+ W_2)(r, \tau , 0, \lambda) $.
\end{lemma}
\begin{proof}
In \cite[Proposition 2.3]{locwill} it was shown that 
$$W_1(r, \tau , 0, \lambda)=  r^2 (2 \lambda - \frac{2}{3}\mathrm{Sc}^\tau(0) + 4 \mathrm{Ric}^\tau_{pq}(0) x^p x^q ) + r^3( 5  \mathrm{Ric}^\tau_{pq,s}(0) x^p x^q x^s  -\mathrm{Sc}^\tau_{,p}(0)  x^p )+ \mathcal{O}(r^4) $$
In the rest of the proof we omit the superindex $\tau$ for simplicity. Now considering the rescaling, we have 
\begin{equation}\label{W2}
\begin{split}
    W_2(r, \tau , \varphi , \lambda)=&r^3 \big(\frac{1}{2}H P^2+P( \nabla_{\nu} \tr k - \nabla_{\nu} k(\nu,\nu ))  - 2P \diver_\Sigma (k(\cdot, \nu)) - 2k (\nabla^\Sigma P, \nu )\big),
    \end{split}
\end{equation}
where the right hand side is evaluated on the geodesic sphere  $ F_\tau (\alpha_r (S^n)):=\Sigma$ using the metric $g$.  Consider a  local frame  $e_i \in TM $ $i=1,2,3$. We use Latin letters as indices to denote the whole frame $ i,j,r,s,t...$ and Greek letters $\alpha, \beta$ just to denote the vectors tangent to $\Sigma$.  We use the Einstein summation convention, and for the sake of simplicity, we omit writing the metric $g^{ij}$ when two indices are contracted. 

First, let us expand  the last two terms of (\ref{W2}).
\begin{equation}
    \begin{split}
        \diver_\Sigma (k(\cdot, \nu)) &= e_\alpha \left(  k(e_\alpha, \nu) \right)= \nabla_{e_\alpha} k(e_\alpha, \nu) + k(\nabla_{e_\alpha} e_\alpha , \nu) + k(e_\alpha, \nabla_{e_\alpha} \nu)\\
        &= \nabla_{e_i} k (e_i, \nu) - \nabla_{\nu} k(\nu, \nu) + k (\nabla_{e_\alpha} e_\alpha , \nu) + g^\Sigma (k, B),
    \end{split}
\end{equation}
where $g^\Sigma (k, B)=  g^{\Sigma\alpha \gamma} g^{\Sigma \beta \sigma } k_{\alpha \beta} B_{\gamma \sigma}$.
\begin{equation}
    \nabla^\Sigma_{e_\alpha} P= e_\alpha (\tr k - k(\nu, \nu))= \nabla_{e_\alpha} k(e_i, e_i) - \nabla_{e_\alpha} k(\nu, \nu) + 2 k(\nabla_{e_\alpha}e_\beta , e_\beta).
\end{equation}
 Now introducing these terms in (\ref{W2}) we have
\begin{equation}\label{W22}
\begin{split}
    W_2(r, \tau , \varphi, \lambda)=&r^3 \Big(\frac{1}{2}H P^2+P( \nabla_{\nu} \tr k - \nabla_{\nu} k(\nu,\nu ))  - 2P\big(\nabla_{e_i} k (e_i, \nu) - \nabla_{\nu} k(\nu, \nu) +  g^\Sigma (k, B)\\ &
    +  k (\nabla_{e_\alpha} e_\alpha , \nu)\big)- 2k_{\alpha j} \nu^j\big( \nabla_{e_\alpha} k(e_i, e_i) - \nabla_{e_\alpha} k(\nu, \nu) + 2 k(\nabla_{e_\alpha}e_\beta , e_\beta)\big) \Big).
    \end{split}
\end{equation}
Now using that for a geodesic sphere, one has $H(r,\tau, 0, \lambda)= \frac{2}{r} - \frac{r^2}{3}  \mathrm{Ric}_{ij}\, x^i x^j- \frac{r^3}{4} \mathrm{Ric}_{ij,l}\, x^i x^j x^l +\mathcal{O}(r^4)    $ (this expression can be found in \cite{Ye}) where $\mathrm{Ric}$ is evaluated at $c(\tau)$,  $B(r,\tau, 0, \lambda)= r^{-1} g^\Sigma + \mathcal{O}(r^2)$,  $\nabla_{\nu} \nu = \mathcal{O}(r^2)$ and taking the frame such that $\nabla_{e_i} e_j = \mathcal{O}(r^2) $.

\begin{equation}\label{W23}
\begin{split}
    W_2(r, \tau , 0, \lambda)=& r^2 P^2+r^3 P( \nabla_{\nu} \tr k - \nabla_{\nu} k(\nu,\nu ))  - 2r^3 P\big(\nabla_{e_i} k (e_i, \nu) - \nabla_{\nu} k(\nu, \nu) - k (\nabla_{\nu} \nu , \nu)\\ &
    + \frac{1}{r} P\big)- 2r^3k(e_j, \nu) \nabla_{e_j} k(e_i, e_i) +2 r^3k(\nu, \nu) \nabla_{\nu} k(e_i, e_i) +2r^3k(e_i, \nu) \nabla_{e_i} k(\nu, \nu)\\&
    -2 r^3k(\nu, \nu) \nabla_{\nu} k(\nu, \nu) + 4r^2 k(e_i, \nu)\,k(e_i, \nu) -4r^2 k(\nu, \nu)\,k(\nu, \nu) +\mathcal{O}(r^4) \\
    =& r^2\big( 4 k(e_i, \nu)\,k(e_i, \nu) -4 k(\nu, \nu)\,k(\nu, \nu) -P^2 \big)+r^3  P\big( \nabla_{\nu} \tr k - \nabla_{\nu} k(\nu,\nu )\\
    &- 2\nabla_{e_i} k (e_i, \nu)+2 \nabla_{\nu} k(\nu, \nu)  \big) +2 r^3 \big(k(\nu, \nu) \nabla_{\nu} k(e_i, e_i) \\&
    - k(e_i, \nu) \nabla_{e_\alpha} k(e_i, e_i)  +k(e_i, \nu) \nabla_{e_i} k(\nu, \nu)
    -k(\nu, \nu) \nabla_{\nu} k(\nu, \nu) \big) +\mathcal{O}(r^4) \\
    =&-r^2 (\tr k)^2 + r^3(\tr k \, \partial_i \tr k - 2\tr k \, k_{i s, s} - 2 k_{s j}\, \partial_s \tr k) x^i  + r^2(2\tr k \, k_{ij}\\ &
    + 4 k_{si} \, k_{sj})x^i x^j +r^3(2 \tr k \,k_{ij,s} - \tr k k_{ij,s} - \partial_i \tr k\, k_{js} +2 k_{ij} \, k_{st,t}\\ &
    +2k_{ij} \, \partial_s \tr k+2 k_{st} \, k_{ij,t} )x^ix^jx^s -5 r^2 k_{ij}\, k_{pq} x^ix^j x^px^q -3r^3 k_{ij} \, k_{pq,s} x^i x^j x^px^q x^s\\
    & +\mathcal{O}(r^4)\\
    =& r^2\big( -(\tr k)^2 + (2\tr k\, k_{ij}+ 4 k_{si} \, k_{sj})x^i x^j  -5  k_{ij}\, k_{pq} x^ix^j x^px^q  \big)\\
    &+ r^3 \big( \big(\frac{\partial_i (\tr k)^2}{2}  - 2\partial_s ( \tr k \, k_{ s i}) \big)x^i    +(\partial_s (\tr k k_{ij}) + 2 \partial_t (k_{ij} k_{ts}) )x^i x^j x^s \\
    &-3 k_{ij} \, k_{pq,s} x^i x^j x^px^q x^s \big) +\mathcal{O}(r^4).
    \end{split}
\end{equation}
 \end{proof}
We have an analogous  result to  \cite[Lemma 3.2]{locwill}.
\begin{lemma}\label{3.2}
For every $\tau \in \mathbb{R}^3 $ and every $ \lambda \in \mathbb{R}$ we have that 
$$\Phi_{\varphi r}(0, \tau , 0, \lambda)=0, $$
where we denote  $ \Phi_{\varphi }(r, \tau , \varphi, \lambda)  \varphi'  = \frac{d}{dt} \Phi(r, \tau,  \varphi +t \varphi' , \lambda)   \vert_{t=0}.$
\end{lemma}

\begin{proof}

First, we consider the terms depending on $k$, that is, expression (\ref{W22}). In \cite[Lemma 1.3]{Ye} it was shown that $H_{\varphi r}(0, \tau , 0, \lambda)=0$ and $B_{\varphi r}(0, \tau , 0, \lambda)=0$, then we have that  the terms of the linearization that don't depend on $B_{\varphi r}$  have order at least $\mathcal{O}(r^2)$ and therefore
$$W_{2\varphi r}(0, \tau , 0, \lambda) = \frac{\partial}{\partial r} W_{2\varphi }(r, \tau , 0, \lambda)_{|r=0} =0. $$
Finally in \cite[Lemma 3.2]{locwill} it was shown that $W_{1\varphi r}(0, \tau , 0, \lambda) =0$ and as $\Phi_{\varphi r}(0, \tau , 0, \lambda)=W_{1\varphi r}(0, \tau , 0, \lambda)+W_{2\varphi r}(0, \tau , 0, \lambda) $ we have the result.
\end{proof}
In \cite[Section 3]{locwill}, it was shown that when $r \to 0$ the linearization of $W_1$ reduces to  
\begin{equation}
    W_{1\varphi }(0, \tau , 0, \lambda) = -\Delta^{\mathbb{S}^2} (- \Delta^{\mathbb{S}^2} -2),
\end{equation}
which is the linearization of the Willmore operator in Euclidean space. The kernel of this operator is generated by the constant functions and the first spherical harmonics, that is $K=\mathrm{Span}  \{1,x^1, x^2, x^3 \}$ where $x^i$ are  coordinate components of a point $x\in \mathbb{S}^2$. Now notice that by our scaling (as seen in Lemma \ref{3.2}) the operator $W_{1\varphi r}(r, \tau , 0, \lambda)$ has order $\mathcal{O}(r^2) $. Therefore, we have 
\begin{equation}
     \Phi_{\varphi }(0, \tau , 0, \lambda) = -\Delta^{\mathbb{S}^2} (- \Delta^{\mathbb{S}^2} -2).
\end{equation}
Now we define precisely what a concentration of surfaces is.
\begin{definition}
 We say that a family of closed compact embedded surfaces $\{S_r: r \in I \}$, where $I$ is an interval satisfying $0 \in \Bar{I} $, is a \emph{concentration of surfaces around $p$} if 
    \begin{equation*} 
        \limsup_{r \to 0} \diam S_r =0
        \quad \text{ and} \quad \bigcap_{r_0 \in (0,\infty)} \overline{\bigcup_{r \in I \cap (0,r_0) }S_r}= \{ p \}.
    \end{equation*}
\end{definition}
Note that a foliation is a concentration of surfaces where the surfaces can be continuously parameterized by $r$ (that is $\forall r\in I$ there is a surface $S_r$) and where the surfaces do not intersect with each other.
\subsection{ Foliation construction}

As mentioned before, if a surface satisfies $\Phi_{\varphi r}(r, \tau , \varphi, \lambda)=0$ then we have an area constrained critical surface of the Hawking functional, then the idea to construct the foliation  is to find by means of the implicit function theorem some $\tau (r)$, $\varphi(r)$ and $\lambda(r)$ such that $\Phi(r,\tau(r), \varphi(r), \lambda(r))=0$ for all $r\in (0,r_0)$. To achieve this,  we use  that we can decompose $\mathcal{C}^{4,\frac{1}{2}}(\mathbb{S}^2)$ as   $K \oplus K^\bot$  where $K$ is the kernel of $-\Delta^{\mathbb{S}^2} (-\Delta^{\mathbb{S}^2} -2)$ on euclidean space  and $K^\bot$ its $L^2$ orthogonal complement. Then if one manages to show that $\Phi(r,\tau(r), \varphi(r), \lambda(r))=0$  holds on $K$ and on $K^\bot$ the equation holds on $\mathcal{C}^{4,\frac{1}{2}}(\mathbb{S}^2)$, and this is precisely  what we are going to show using the implicit function theorem in each of the cases.

\begin{theorem}\label{primfoli}
Let $p \in M $ be such that at $p$, $\nabla(  \mathrm{Sc} + \frac{3}{5} (\tr k)^2 + \frac{1}{5} |k|^2)=0  $ and $\nabla^2(  \mathrm{Sc} + \frac{3}{5} (\tr k)^2 + \frac{1}{5} |k|^2) $ is nondegenerate. Then there exist  $\delta, \epsilon_0, C >0$ such that if at $p$, 
\begin{equation}\label{folicon}
 C |(\nabla^2(  \mathrm{Sc} + \frac{3}{5} (\tr k)^2 + \frac{1}{5} |k|^2))^{-1}| \cdot \; |k|\; |\nabla k|\, (|k|^2 + |\mathrm{Ric}|  ) <1,
 \end{equation} then there exist a smooth foliation $\mathcal{F}= \{ S_r : r\in (0, \delta)  \} $  around $p$ of  area constrained critical spheres of the Hawking functional, that is surfaces satisfying equation (\ref{eulag}), for some $\lambda \in \mathbb{R}$. Furthermore, these surfaces can be express  as normal graphs over geodesic spheres of radius $r$, and they  satisfy $\mathcal{H}(S_r) < 4\pi +\epsilon_0^2$ and $|S_r|< \epsilon_0^2 $,  for $r \in (0, \delta)$.
\end{theorem}

\begin{proof}
We split the kernel K in two parts $K_0=\mathrm{Span }\{ 1\} $ and  $K_1=\mathrm{Span}\{ x^1, x^2, x^3\} $.
Let $\pi_i$ for $i=0,1$ denote the orthogonal projection from $ \mathcal{C}^{0, \frac{1}{2}}( \mathbb{S}^n)$ onto $K_i$, let $T_1:K_1 \rightarrow  \mathbb{R}^3 $ be the isomorphism sending $x^i_{| \mathbb{S}^2} $  to the $i$th coordinate basis $e_i$, and let $T_0:K_0 \rightarrow  \mathbb{R} $ be the identity map. Define $ \Tilde{\pi}_i:= T_i \circ \pi_i$ for $i=1,2$. We expand the operator
\begin{equation}
\begin{split}
    \Phi(r, \tau, r^2 \varphi, \lambda )=&  \int_0^1 \frac{\partial}{\partial t} (\Phi(r, \tau,t r^2 \varphi, \lambda )) dt + \Phi(r, \tau, 0, \lambda )  \\
    =& \int_0^1 \int_0^1 \frac{\partial}{\partial s} (\Phi_{\varphi}(sr, \tau,st r^2 \varphi, \lambda )) ds r^2 \varphi dt +  \Phi(r, \tau, 0, \lambda ) \\
    &+ \Phi_\varphi(0, \tau, 0, \lambda )\varphi^2r^2
    \end{split}
\end{equation}
and continuing the same procedure, we obtain
\begin{equation}\label{expan}
    \begin{split}
        \Phi(r, \tau, r^2 \varphi, \lambda )=& \Phi(r, \tau, 0, \lambda ) +\Phi_\varphi (0, \tau, 0, \lambda) \varphi r^2 +\Phi_{\varphi r}(0, \tau , 0, \lambda)\varphi r^3 \\
        &+ r^4 \int_0^1 \int_0^1 t \Phi_{\varphi \varphi}(sr, \tau,st r^2 \varphi, \lambda )  \varphi \varphi ds dt\\
        &+ r^4 \int_0^1 \int_0^1 \int_0^1 s \Phi_{\varphi r r}(usr, \tau,ust r^2 \varphi, \lambda )   \varphi du ds dt\\
        &+ r^5 \int_0^1 \int_0^1 \int_0^1 s t \Phi_{\varphi \varphi r}(usr, \tau,ust r^2 \varphi, \lambda )  \varphi \varphi du ds dt.\\
    \end{split}
\end{equation}
Note that $\Phi_{\varphi r}(0, \tau , 0, \lambda)\varphi =0$ by Lemma \ref{3.2}. We will study the projection of this expansion to the kernel. We have for the first term that  $\Phi(r, \tau, 0, \lambda )= W_1(r, \tau, 0, \lambda )+ W_2(r, \tau, 0, \lambda )$ and   in \cite[Lemma 3.1]{locwill} it was shown that
\begin{equation}\label{willker}
    \begin{split}
        \Tilde{\pi}_0 \left(  W_1(r, \tau, 0, \lambda )\right)=& 8 \pi r^2 \left( \lambda + \frac{1}{3} \mathrm{Sc}^\tau(0)  \right) + \mathcal{O}(r^4)\\
         \Tilde{\pi}_1\left(  W_1(r, \tau, 0, \lambda )\right)=& \frac{4 \pi}{3}r^3 \nabla_{e_i}  \mathrm{Sc}^\tau(0) e_i + \mathcal{O}(r^5)
    \end{split}
\end{equation}
Now using equation (\ref{sphep}) and the fact that $\int_{\mathbb{S}^2} x^i d\mu = \int_{\mathbb{S}^2} x^i x^jx^p d\mu =\int_{\mathbb{S}^2} x^ix^j x^p x^q x^s d \mu =0 $ we have
\begin{equation}
    \begin{split}
       \Tilde{\pi}_0\left(  \frac{W_2(r, \tau, 0, \lambda )}{r^2}  \right)_{|r=0}
        =& \int_{\mathbb{S}^2}\big(   \left(
     2\tr k^\tau(rx)\, k^\tau_{ij}(rx)+ 4 k^\tau_{si}(rx) \, k^\tau_{sj}(rx)\right)x^i x^j \\
    &-(\tr k^\tau)(rx)^2 -5  k^\tau_{ij}(rx)\, k^\tau_{pq}(rx) x^ix^j x^px^q   \big)d\mu_{|r=0}\\
    =& \left( 2 \tr k^\tau(0)\, k^\tau_{ij}(0)+ 4 k^\tau_{si}(0) \, k^\tau_{sj}(0)\right) \int_{\mathbb{S}^2}x^i x^j d\mu \\
    &-(\tr k^\tau)(0)^2\int_{\mathbb{S}^2}d\mu -5  k^\tau_{ij}(0)\, k^\tau_{pq}(0) \int_{\mathbb{S}^2}x^ix^j x^px^q  d\mu\\
    =& 8 \pi  ( -\frac{2}{3}(\tr k^\tau)^2 + \frac{2}{3} |k^\tau|^2)
    \end{split}
\end{equation}
where Lemma \ref{integrals} was used and  the quantities are evaluated at the point $c(\tau)$.

Note that for any $\varphi_0 \in K^\bot$ one has $\Tilde{\pi}_i(\Phi_\varphi (0, \tau, 0, \lambda) \varphi ) =0 $, then taking some arbitrary $\varphi_0 \in K^\bot$ which will be fixed later,   and $\lambda_0 = - \frac{1}{3}\mathrm{Sc} +\frac{2}{3}(\tr k^\tau)^2 - \frac{2}{3} |k^\tau|^2 $  where the geometric quantities are evaluated at $p$, we find  using the expansion (\ref{expan}) that
\begin{equation}\label{k0}
    \begin{split}
              \Tilde{\pi}_0\left( \frac{\Phi(r,\tau,  r^2\varphi,\lambda )}{r^2} \right)_{|\substack{r=0,\tau=0, \lambda=\lambda_0, \varphi=\varphi_0}}  = 8 \pi (  \lambda_0 + \frac{1}{3} \mathrm{Sc}^\tau -\frac{2}{3}(\tr k^\tau)^2 + \frac{2}{3} |k^\tau|^2)_{|\tau=0}   =0
    \end{split}
\end{equation}
Using again the expansion (\ref{expan}) and (\ref{willker}) we have 
\begin{equation}\label{proj1}
    \begin{split}
              \Tilde{\pi}_1\left( \frac{\Phi(r,\tau,  r^2\varphi_0,\lambda )}{r^3} \right)_{|\substack{r=0,\tau=0, \lambda=\lambda_0}} &= \frac{4 \pi}{ 3}  \mathrm{Sc}_{,i}e_i +  \Tilde{\pi}_1 \left(  \frac{ W_2(r, \tau, 0, \lambda )}{r^3}  \right)_{|\substack{r=0,\tau=0, \lambda=\lambda_0}}\\
              &= \frac{4 \pi}{ 3}  \mathrm{Sc}_{,i}e_i +  \Tilde{\pi}_1 \Big(    \big(\frac{\partial_i (\tr k^\tau)^2}{2}  - 2\partial_s ( \tr k^\tau \, k^\tau_{ s i}) \big) x^i   +(\partial_s (\tr k^\tau k^\tau_{ij}) \\
    &+ 2 \partial_t (k^\tau_{ij} k^\tau_{ts}) )x^i x^j x^s -3 k^\tau_{ij} \, k^\tau_{pq,s} x^i x^j x^px^q x^s   \Big)_{|\substack{r=0,\tau=0, \lambda=\lambda_0}} \\
    &+\frac{1}{r} \Tilde{\pi}_1 \big(   (6 \tr k^\tau\, k^\tau_{ij}+ 4 k^\tau_{si} \, k^\tau_{sj})x^i x^j-(\tr k^\tau)^2\\
    &-9  k^\tau_{ij}\, k^\tau_{pq} x^ix^j x^px^q    \big)_{|\substack{r=0,\tau=0, \lambda=\lambda_0}}\\
    \end{split}
\end{equation}
Let's see in detail the last two terms of this expression, we have that the second term is equal to 
\begin{equation}
    \begin{split}
                 \int_{\mathbb{S}^2} &\big(\frac{\partial_i (\tr k^\tau)^2}{2}  - 2\partial_s ( \tr k^\tau \, k^\tau_{ s i}) \big) x^i x^l  +(\partial_s (\tr k^\tau k^\tau_{ij}) 
    + 2 \partial_t (k^\tau_{ij} k^\tau_{ts}) )x^i x^j x^sx^l \\&
    -3 k^\tau_{ij} \, k^\tau_{pq,s} x^i x^j x^px^q x^sx^l  d\mu _{|\substack{r=0,\tau=0}} e_l\\
    =&  \big(\frac{\partial_i (\tr k^\tau)^2}{2}  - 2\partial_s ( \tr k^\tau \, k^\tau_{ s i}) \big) \int_{\mathbb{S}^2}x^i x^l d\mu +(\partial_s (\tr k^\tau k^\tau_{ij}) 
    + 2 \partial_t (k^\tau_{ij} k^\tau_{ts}) ) \int_{\mathbb{S}^2}x^i x^j x^sx^l d\mu\\
    &-3 k^\tau_{ij} \, k^\tau_{pq,s} \int_{\mathbb{S}^2} x^i x^j x^px^q x^sx^l  d\mu  e_l\\
    =& \frac{92\pi}{105} \partial_l (\tr k)^2 e_l- \frac{64\pi}{35} \partial_s(\tr k\, k_{sl})e_l + \frac{64\pi}{105} \partial_t( k_{ls} \,k_{st})e_l - \frac{12 \pi}{105} \partial_l |k|^2e_l
    \end{split}
\end{equation}

For the last term of the expression note that $\Tilde{\pi}_1 \big(  ( 6\tr k^\tau \, k^\tau_{ij}
    + 4 k^\tau_{si} \, k^\tau_{sj})x^i x^j
    -(\tr k^\tau)^2 \\-9  k^\tau_{ij}\, k^\tau_{pq} x^ix^j x^px^q   \big)_{|\substack{r=0}}=0
    $ and that $\frac{\partial}{\partial r}_{|r=0} k_{ij}(rx)= k_{ij,t}(0)x^t$, then by performing a Taylor expansion around $r=0$ we find
   \begin{equation*}
    \begin{split}
    &\Big(\frac{1}{r} \Tilde{\pi}_1 \big(   (6 \tr k^\tau\, k^\tau_{ij}+ 4 k^\tau_{si} \, k^\tau_{sj})x^i x^j
    -(\tr k^\tau)^2 -9  k^\tau_{ij}\, k^\tau_{pq} x^ix^j x^px^q    \big) \Big)_{|\substack{r=0,\tau=0, \lambda=\lambda_0}}\\
    &=\frac{\partial}{\partial r} \int_{\mathbb{S}^2} \big((6\tr k^\tau\, k^\tau_{ij}+ 4 k^\tau_{si} \, k^\tau_{sj})x^i x^j
    -(\tr k^\tau)^2 -9  k^\tau_{ij}\, k^\tau_{pq} x^ix^j x^px^q\big)x^l  d\mu_{|\substack{r=0,\tau=0, \lambda=\lambda_0}}  e_l\\
    &=\int_{\mathbb{S}^2} \big((6 \partial_p(\tr k, k_{ij})+ 4 \partial_p (k_{si} \, k_{sj}))x^i x^jx^p
    -\partial_i (\tr k)^2x^i -9  \partial_s (k_{ij}\, k_{pq}) x^ix^j x^px^qx^s\big)x^l  d\mu \, e_l\\ 
    &=-\frac{8\pi}{105} \partial_l (\tr k)^2 e_l+ \frac{64\pi}{35} \partial_s(\tr k\, k_{sl})e_l - \frac{64\pi}{105} \partial_t( k_{ls} \,k_{st})e_l + \frac{8 \pi}{21} \partial_l |k|^2e_l
    \end{split}
\end{equation*}
Then putting everything back into (\ref{proj1}), we obtain
\begin{equation}\label{proj2}
    \begin{split}
              \Tilde{\pi}_1\left( \frac{\Phi(r,\tau,  r^2\varphi_0,\lambda )}{r^3} \right)_{|\substack{r=0,\tau=0, \lambda=\lambda_0}} &= \frac{4 \pi}{ 3} \partial_l( \mathrm{Sc}+ \frac{3}{5}  (\tr k)^2 + \frac{1}{5} |k|^2  )e_l =0.  
    \end{split}
\end{equation}

 To apply the implicit function theorem for the system of equations (\ref{k0}) and (\ref{proj2}),  we need the corresponding operator to be invertible. Let us find the operator. We  compute the following derivatives 
\begin{equation*}
    \begin{split}
              \frac{\partial}{\partial \lambda }\Tilde{\pi}_0\left( \frac{\Phi(r,\tau,  r^2\varphi,\lambda )}{r^2} \right)_{|\substack{r=0,\tau=0, \lambda=\lambda_0}} =8 \pi, \quad \frac{\partial}{\partial \lambda } \Tilde{\pi}_1\left( \frac{\Phi(r,\tau,  r^2\varphi,\lambda )}{r^3} \right)_{|\substack{r=0,\tau=0, \lambda=\lambda_0}} = 0,
    \end{split}
\end{equation*}
\begin{equation*}
    \begin{split}
              \frac{\partial}{\partial \tau^\beta }\Tilde{\pi}_0\left( \frac{\Phi(r,\tau,  r^2\varphi,\lambda )}{r^2} \right)_{|\substack{r=0,\tau=0, \lambda=\lambda_0}} =\frac{8 \pi}{3} (    \partial_{\tau^\beta} \mathrm{Sc} + \frac{1}{5} \partial_{\tau^\beta}|k|^2 + \frac{3}{5} \partial_{\tau^\beta}(\tr k)^2)   )=0,
    \end{split}
\end{equation*}
\begin{equation*}
    \begin{split}
              \frac{\partial}{\partial \tau^\beta } \Tilde{\pi}_1\left( \frac{\Phi(r,\tau,  r^2\varphi,\lambda )}{r^3} \right)_{|\substack{r=0,\tau=0, \lambda=\lambda_0}} &= \frac{4 \pi}{ 3} \partial_{\tau^\beta} \partial_l( \mathrm{Sc}+ \frac{3}{5}  (\tr k)^2 + \frac{1}{5} |k|^2  )e_l.
    \end{split}
\end{equation*}
Then  we need the operator 
\begin{equation}
    \begin{pmatrix}
 8\pi & 0  \\
 0 & \frac{4 \pi}{ 3} \nabla^2( \mathrm{Sc}+ \frac{3}{5}  (\tr k)^2 + \frac{1}{5} |k|^2  )
\end{pmatrix}
\end{equation}
 to be invertible  at point $p$ and this is equivalent to have $\nabla^2( \mathrm{Sc}+ \frac{3}{5}  (\tr k)^2 + \frac{1}{5} |k|^2  )$ invertible. Then there exist functions $\tau =\tau(r,\varphi)$ and $\lambda =\lambda(r, \varphi)$ such that $\tau(0,\varphi_0)=0$, $\lambda(0,\varphi_0)=\lambda_0=- \frac{1}{3}\mathrm{Sc} -\frac{2}{3} |k|^2 + \frac{2}{3} (\tr k)^2 $  and  $\Tilde{\pi}_i( \Phi(r,\tau,  r^2\varphi,\lambda ))=0$ $i=1,2 $ for $(r,\tau,\varphi,\lambda)$ close to $(0,0,\varphi_0, \lambda_0)$.

Now let us  apply the implicit function theorem to have a vanishing projection to the orthogonal to the kernel. First, we fix the map $\varphi_0 \in K^\bot$ to be the solution to the equation 
\begin{equation}\label{phi_0}
    -\Delta^{\mathbb{S}^2} (- \Delta^{\mathbb{S}^2} -2) \varphi_0= \pi^\bot\left( 9  k^\tau_{ij}\, k^\tau_{pq} x^ix^j x^px^q - (4 \mathrm{Ric}_{ij}+  6 \tr k^\tau\, k^\tau_{ij}+ 4 k^\tau_{si} \, k^\tau_{sj})x^i x^j
      \right) 
\end{equation}
where $ \pi^\bot$ is the orthogonal  projection to $K^\bot$. Then we obtain projecting (\ref{expan}) to $K^\bot$ and normalizing it by $r^2$.
\begin{equation}\label{kort}
\begin{split}
  \pi^\bot \left(  \frac{\Phi(r,\tau,  r^2\varphi,\lambda )}{r^2} \right)_{|r=0,\varphi=\varphi_0}=&\pi^\bot \big( -2(\frac{1}{3}\mathrm{Sc} + \frac{1}{15} |k|^2 + \frac{1}{5} (\tr k)^2  ) -\frac{2}{3}\mathrm{Sc}  +4 \mathrm{Ric}_{ij}x^ix^j    \\ &  -(\tr k^\tau)^2
   + (6\tr k^\tau \, k^\tau_{ij}
    + 4 k^\tau_{si} \, k^\tau_{sj})x^i x^j  -9  k^\tau_{ij}\, k^\tau_{pq} x^ix^j x^px^q\\
    &- \Delta^{\mathbb{S}^2}(\Delta^{\mathbb{S}^2} -2) \varphi_0   \big)\\
  =&  \pi^\bot\left(  (4 \mathrm{Ric}_{ij}+  6 \tr k^\tau\, k^\tau_{ij}+ 4 k^\tau_{si} \, k^\tau_{sj})x^i x^j
       -9  k^\tau_{ij}\, k^\tau_{pq} x^ix^j x^px^q \right)\\
      &-\Delta^{\mathbb{S}^2} (- \Delta^{\mathbb{S}^2} -2) \varphi_0\\
      =&0
\end{split}
\end{equation}
\begin{equation}
    \frac{\partial}{\partial \varphi}\pi^\bot \left(  \frac{\Phi(r,\tau,  r^2\varphi,\lambda )}{r^2} \right)_{|r=0,\varphi=\varphi_0}= -\Delta^{\mathbb{S}^2} (- \Delta^{\mathbb{S}^2} -2)\vert_{K^\bot}  
\end{equation}
and this operator is invertible since our equation is restricted to $K^\bot$ (the $K$ part is zero). Then by the implicit function theorem, there exist some $\delta>0$, $\tau=\tau(r)$, $\varphi (x)=\varphi(x,r)$ and $\lambda =\lambda(r) $ such that $\Phi(r, \tau(r), r^2 \varphi(r), \lambda(r) )=0  $ for $0<r<\delta $, this means that for each $r$ we have an area constrained critical surface of the Hawking functional. Now let's see that these surfaces form a foliation.

 By construction, we have the following parametrization for our surfaces.
 \begin{equation}\label{parame}
     G:\mathbb{R}^+ \times \mathbb{S}^n \mapsto M, \quad (r,x) \mapsto \exp_{c(\tau(r))}\big(rx(1+r^2 \varphi(r))\big)  
 \end{equation}
 
where we write $\varphi(r)=\varphi(r)(x) $ for simplicity. To find the lapse function of  these surfaces one calculates
 
$ \frac{\partial G }{\partial r}_{|r=0} = \left(d_x \exp_{c(\tau(r))} \right) \big(x(1 +r^2 \varphi(r)) + rx(r^2 \varphi(r))_r\big)_{|r=0} + \left( \frac{\partial  \exp_{c(\tau(r))}}{\partial r}  \right) \big( rx(1 +r^2 \varphi(r))\big)_{|r=0} $

 and this reduces to $ \frac{\partial G}{\partial r}_{|r=0}= x + \frac{\partial \tau^k}{\partial r}_{|r=0} e_k  $,
then we see that the lapse function is given by 
    \begin{equation}\label{lapse}
        \alpha:=\langle \frac{\partial G}{\partial r}_{|r=0},\nu  \rangle = 1 +  \frac{\partial \tau^k}{\partial r}\langle   e_k, \nu \rangle 
    \end{equation}
     therefore we have a foliation if  $\alpha>0$, then it suffices to show that $|\frac{\partial\tau }{\partial r}_{|r=0}|<1 $.  To estimate $\frac{\partial\tau }{\partial r}_{|r=0} $ we will use that the equation $ \Tilde{\pi}_1\left(\frac{\Phi(r, \tau, r^2 \varphi, \lambda )}{r^3}\right)=0 $ implies that  $ \frac{\partial}{\partial r} \Tilde{\pi}_1\left(\frac{\Phi(r, \tau, r^2 \varphi, \lambda )}{r^3}\right)_{|r=0}=0$ and by (\ref{expan})  this is  
\begin{equation}\label{estau}
\begin{split}
    0=& \frac{\partial}{\partial r} \Tilde{\pi}_1\left(\frac{\Phi(r, \tau, 0, \lambda )}{r^3}\right)_{|r=0}  + \frac{1}{2} \Tilde{\pi}_1\left( \Phi_{\varphi \varphi}(0, 0,0,0 ) \varphi_0 \varphi_0  \right) +\frac{1}{2} \Tilde{\pi}_1\left( \Phi_{\varphi r r}(0, 0,0,0 ) \varphi_0   \right).\\
    \end{split}
\end{equation}
Note that the second term is equal to zero. For the first, term it is not hard to see using (\ref{proj2}) and the chain rule that 
\begin{equation}
    \frac{\partial}{\partial r} \Tilde{\pi}_1\left(\frac{\Phi(r, \tau, 0, \lambda )}{r^3}\right)_{|r=0}= \frac{4 \pi}{ 3} \partial_{\tau^\beta} \partial_l( \mathrm{Sc}+ \frac{3}{5}  (\tr k)^2 + \frac{1}{5} |k|^2  )\frac{\partial\tau^\beta }{\partial r}_{|r=0}e_l
\end{equation}

then from (\ref{estau}) and the invertibility of $\nabla^2(  \mathrm{Sc} + \frac{3}{5} (\tr k)^2 + \frac{1}{5} |k|^2) $ we have
\begin{equation}\label{foli2}
     |\frac{\partial\tau }{\partial r}_{|r=0}| < \frac{3}{4 \pi} |(\nabla^2(  \mathrm{Sc} + \frac{3}{5} (\tr k)^2 + \frac{1}{5} |k|^2))^{-1}| \cdot | \frac{1}{2} \Tilde{\pi}_1\left( \Phi_{\varphi r r}(0, 0,0,0 ) \varphi_0   \right)|
\end{equation}
In the following, we show that the right hand side of the previous
expression is less than one.  The solution of the equation (\ref{phi_0}) is a function of the form  $\varphi_0 =(k\ast k)_{ijpq} \;x^i x^j x^p x^q  + C\cdot ( \mathrm{Ric} +  k\ast k )_{ij} \; x^i x^j   + C \cdot (\mathrm{Sc} +  k\ast k )$, where we denote  for any tensors $A$ and $B,$ $A \ast B$ to be any linear combination of contractions of $A$ and $B$ with the correspondent metric. In particular, we have that $\varphi_0$ is an even function. In \cite[Lemma 4.1]{locwill}, it was shown that $W_{1\varphi r r}(0, 0,0,0 ) $ is an even operator which implies that $\Tilde{\pi}_1\left( W_{1\varphi r r}(0, 0,0,0 ) \varphi_0   \right) =0$.   Unfortunately  the operator $W_{2\varphi r r}(0, 0,0,0 ) $ is not even, it has an odd part which is proportional to  $\nabla k \ast k  $, then combining this with the expression of $\varphi_0$ in (\ref{foli2}) we obtain the estimate 
$$   |\frac{\partial\tau }{\partial r}_{|r=0}| < C |(\nabla^2(  \mathrm{Sc} + \frac{3}{5} (\tr k)^2 + \frac{1}{5} |k|^2))^{-1}| \cdot \; |k|\; |\nabla k|\, (|k|^2 + |\mathrm{Ric}|  ) $$
where $C$ depends on $n$. Then if $ |(\nabla^2(  \mathrm{Sc} + \frac{3}{5} (\tr k)^2 + \frac{1}{5} |k|^2))^{-1}| \cdot \; |k|\; |\nabla k|\, (|k|^2 + |\mathrm{Ric}|  ) $  is small enough we have $|\frac{\partial \tau}{ \partial r}_{| r=0}| <1$ and in particular a foliation.

The leaves of the foliation are normal graphs of the map $r^3 \varphi(r)$  over geodesics spheres of radius $r$. This implies that the mean curvature of our surfaces  can be estimated by the mean curvature   of the geodesic sphere and $\Hess \varphi_0$. Then using that $|| \varphi||_{\mathcal{C}^2} < C$ with $C$ depending on the value of $\mathrm{Ric} $ and $k$ in these coordinates at $p$ we have
$$|H_{S_r}|<|H_{F_\tau(\mathbb{S}_r^n)} | + \mathcal{O}(r^2) < \frac{2}{r} + \mathcal{O}(r)  $$
Then proceeding in the same way as it was done in \cite[Lemma 5.1]{ikoma}, we find that the Willmore energy of the surfaces satisfy 
$$ \frac{1}{4} \int_{S_r} H^2  d\mu  = 4 \pi +  \mathcal{O}(r^2) $$
and $|S_r| = 4 \pi r^2+ \mathcal{O}(r^4)$, then it is direct to see that there exists an $\epsilon_0$ such that $$\mathcal{H}(\Sigma)= \frac{1}{4} \int_{S_r} H^2 - P^2 d\mu < 4 \pi + \epsilon_0^2$$ and $|S_r| < \epsilon_0^2 $ for any $r \in (0, \delta)$. Note that the smaller $\delta$ is, the smaller $\epsilon_0$ can be.
\end{proof}
\begin{remark}
$(i)$ Note that condition (\ref{folicon}) is a sufficient but not a necessary condition to have the foliation. The necessary condition is that $\alpha= 1 +  \frac{\partial \tau^k}{\partial r}_{|r=0} \langle   e_k, \nu \rangle >0$, if this condition is not fulfilled, then we only have a regularly centered concentration of critical surface of the Hawking functional around $p$. 

$(ii)$ Note that any initial data set with a local minimum or maximum for the function $ \mathrm{Sc} + \frac{3}{5} (\tr k)^2 + \frac{1}{5} |k|^2$ has a concentration of such surfaces. In particular, any  compact initial data set has at least two. 
\end{remark}

\subsection{Uniqueness and nonexistence}
Now we prove that a point possessing a foliation of area constrained critical surfaces of the Hawking energy cannot have any other of such foliations. That is, the previously constructed foliation is unique.
\begin{theorem}\label{uniq}
$(i)$ Assume that at $p$ $\nabla(  \mathrm{Sc} + \frac{3}{5} (\tr k)^2 + \frac{1}{5} |k|^2)=0  $, $\nabla^2(  \mathrm{Sc} + \frac{3}{5} (\tr k)^2 + \frac{1}{5} |k|^2) $ is nondegenerate and that the foliation $ \mathcal{F}$ of Theorem \ref{primfoli}  exists satisfying  $\mathcal{H}(\Sigma) < 4\pi +\epsilon_0^2$ and $|\Sigma|< \epsilon_0^2 $ for any $\Sigma \in \mathcal{F} $ and the $\epsilon_0$ of the theorem. If  $ \mathcal{F}_2$ is a  foliation around $p$ of  area constrained critical spheres of the Hawking functional, which satisfy $\mathcal{H}(\Sigma) < 4\pi +\epsilon^2$ and $|\Sigma|< \epsilon^2 $ for any $\Sigma \in \mathcal{F}_2 $ and some $\epsilon \leq \epsilon_0$, then either $\mathcal{F}$ is a restriction of $\mathcal{F}_2$ or $\mathcal{F}_2$ is a restriction of $\mathcal{F}$.

$(ii)$  Claim $(i)$  also holds if, instead of  foliations, we consider a concentration of  surfaces around $p$ that satisfy $\mathcal{H}(\Sigma) < 4\pi +\epsilon^2$ and $|\Sigma|< \epsilon^2 $ for any $\Sigma \in \mathcal{F}_2 $ and  $\epsilon \leq \epsilon_0$.
\end{theorem}
\begin{proof}
The idea of the proof is to show that the leaves of the foliation can be expressed as normal graphs over geodesic spheres. Once this is done, we obtain the uniqueness of the foliation  from the implicit function theorems used in Theorem \ref{primfoli}.

Consider the leaves of the foliation $\mathcal{F}_2$ being parametrized by their area radius, that is, $S_r \in \mathcal{F}_2$ where $r$ satisfies $ |S_r| = 4 \pi r^2$, and we consider  $r$ so small that the leaves are contained in a small geodesic sphere where we have a decomposition of the metric as in (\ref{normalcor}). By assumption, the leaves  satisfy $\mathcal{H}(S_r) < 4\pi +\epsilon^2$ and $|S_r|< \epsilon^2 $. Therefore,  by considering $r$ smaller if necessary, we can apply directly  \cite[Proposition 3.2,   Corollary 3.3]{Alex}, obtaining that the surfaces satisfy
\begin{equation}\label{estimla}
  \int_{S_r} |\nabla^2 H|^2+ H^2 |\nabla H|^2 +H^2 |\nabla \mathring{B}|^2 + H^4|\mathring{B}|^2 d\mu <C,  
\end{equation}
\begin{equation}\label{estimsho}
    ||\mathring{B}||_{L^2(S_r)} < C| S_r|, \, \quad \, \left|\left|H - \frac{2}{r} \right|\right|_{L^\infty(S_r)} < C |S_r|^{\frac{1}{2}},
\end{equation}
where the $C$'s are constants depending on the injectivity radius of $p$, $\epsilon$ and of the value of $Ric$, $\nabla Ric$ at $p$. Note also that by  using (\ref{estimla}), (\ref{estimsho})  and Lemma \ref{Sobol} one can reproduce the proof of \cite[Lemma 2.10]{Janme} in the exact same way obtaining the estimate 
$$ ||\mathring{B}||_{L^\infty(S_r)} \leq C r.$$
From (\ref{estimsho})  and by considering $r$ small enough,  we can apply Lemma \ref{refine2}, obtaining 
\begin{equation}\label{normal2}
 \big|\big|\frac{y}{r} -\nu \big|\big|_{L^2(S_r)} <Cr^3,
\end{equation}
where $y$ denotes the position vector on some normal coordinates centered at a point $p_0$. To see that we can express our leaves as graphs over  geodesic spheres we need the normal $\nu$ to $S_r$,   to satisfy on euclidean space that  $\langle \nu, \frac{y}{r}  \rangle \neq 0$, and this is true if we have that $||\frac{y}{r} -\nu ||_{L^\infty (S_r)} $ is small. 
For any tangent vector $e_i$ to $S_r$ and its tangential projection to a sphere of radius $r$ in euclidean space $e_i^T = e_i- \delta(e_i, \frac{y}{r}) \frac{y}{r}$, we have
$$\nabla^E_{e_i} \frac{y}{r}=\frac{1}{r}\left( e_i- \delta(e_i, \frac{y}{r}) \frac{y}{r}\right) \quad \text{and} \quad \nabla_{e_i} \nu =\frac{1}{2} H e_i + \mathring{B}(e_i, \cdot) $$

then by using that $\delta(e_i, \frac{y}{r}) = (\delta-g)(e_i ,\frac{y}{r}) + g(e_i,\frac{y}{r} -\nu ) $ and the decay of the metric $g$ (like in Lemma \ref{comparison}) we obtain 
\begin{equation}\label{devire}
    \big|\nabla \big(\nu -\frac{y}{r}\big) \big| < C\big(|\partial g| + \big|H -\frac{2}{r}\big| +|\mathring{B}| +r^{-1}\big(|g- \delta| + \big|\frac{y}{r} -\nu\big| \big)\big)< Cr +Cr^{-1} \big|\frac{y}{r} -\nu\big|
\end{equation}
for some constant $C$.  From this inequality and (\ref{normal2}), we obtain $ ||\nabla (\frac{y}{r} -\nu) ||_{L^2(S_r)}  <C r^2 $, then using the inequality  (\ref{Sobol2}) from Lemma \ref{Sobol} with $p=2$ we obtain $||\frac{y}{r} -\nu ||_{L^4(S_r)} <Cr^\frac{5}{2} $, now using (\ref{devire}) again  we have  $||\nabla (\frac{y}{r} -\nu) ||_{L^4(S_r)} <Cr^\frac{3}{2} $. Finally, using the Sobolev inequality (\ref{Sobolev2}) for $p=4$ we obtain 
$$\big|\big|\frac{y}{r} -\nu \big|\big|_{L^\infty (S_r)} < C r^2.  $$
Then for $r$ small enough, we can express $S_r$ as a graph over a geodesic sphere of radius $ \Tilde{r}=\Tilde{r} (r)$ centered on a point $p_r$, then we can also characterize the leaves by this radius and denote them by $S_{\Tilde{r}}$. Let us change the notation and simply denote  $\Tilde{r}$ by $r$. Then we have $S_r= F_{\Tilde{\tau}(r)} (\alpha_r (S_{\Tilde{\varphi}}))$ for some $\Tilde{\varphi} \in \mathcal{C}^{4, \frac{1}{2}}(\mathbb{S}^2)$ and  $\Tilde{\tau}(r)$ which satisfies  $\Tilde{\tau}(r) \to 0 $ as $r \to 0$ and $c(\Tilde{\tau}) = \exp_p (\Tilde{\tau}^i e_i)$  where we used the notation of (\ref{coordinates}). 
    
Denoting by $\mathbb{S}^2(a)$ the unit sphere of center $a$ in $\mathbb{R}^3$, $ S_\varphi(a) :=\{x+ \varphi(x) \nu(x): x \in \mathbb{S}^2(a) \}$ and defining  $\Tilde{S}_r :=\alpha_{1/r}(F^{-1}_0(S_r))$ with Euclidean center of mass denoted by $x(r)$, we have that   the previous is equivalent to have $\Tilde{S}_r=S_{\Bar{\varphi}(r)}(x(r)) $ for some smooth function  $\Bar{\varphi}(r)$ on $\mathbb{S}^2(a)$. Furthermore, by Theorem \ref{muller}, our surfaces approach uniformly a round sphere in Euclidean space as $r \to 0$. Hence, in particular,  we obtain that $||\Bar{\varphi}(r) ||_{\mathcal{C}^5} \to 0$ as $r \to 0$.  Observing that this matches the main conclusion of \cite[Lemma 2.3]{Ye}, we can now apply the subsequent results, specifically  \cite[Corollary 2.1 and Lemma 2.4]{Ye}, directly to our setting. These results tell us that we can perturb the center of our spheres with a smooth function $a(r)$ with $ a(r)\in \mathbb{R}^3$ and $\lim_{r \to 0} ||a(r)|| =0$,  so that we can express our surfaces as  $S_r=F_{r(x(r) +a(r))} (\alpha_r (S_{\varphi(r,a(r))})) $, where $\varphi(r,a(r)) $ is some smooth function  on $\mathbb{S}^2$ which satisfies $\pi_1(\varphi(r,a(r)))=0 $ and that $||\varphi(r,a(r)) ||_{\mathcal{C}^5} \to 0$ as $r \to 0$. 
    
    We want  $\varphi$ to satisfy the same conditions as  in Theorem \ref{primfoli}, ensuring we can apply the implicit function theorem’s uniqueness result. In particular, we also require $ \pi_0(\varphi(r,a(r)))=0$. To achieve this, we will need to perturb the radius of our spheres.

    Denote by $m(\varphi(r)):= \pi_0(\varphi(r,a(r)) )= \frac{1}{4 \pi} \int_{\mathbb{S}^2} \varphi (r,a(r)) d \mu $ and note that $ m(\varphi(r)) \to 0  $ for $r \to 0$, then define
    \begin{equation}
        \varphi^* (r):= \frac{\varphi(r,a(r))- m(\varphi(r))}{ 1 +  m(\varphi(r))}  \quad \text{and \quad}  r^*(r):=r(1+ m(\varphi(r))).
    \end{equation}
    We then have $ \pi (\varphi^* (r))=0$ and as $r^* x(1+ \varphi^*(r) ) = r  x(1+ \varphi(r,a(r)))$ for $ x \in \mathbb{S}^2$  then $$S_r  =  F_{\tau(r)} (\alpha_r (S_{\varphi(r,a(r))})) =F_{\tau(r)} (\alpha_{r^*} (S_{\varphi^*(r)})), $$ where $\tau(r) = r(x(r) +a(r)) $. As $ r^* \to 0$ for $r\to 0$ and for $r$ small enough the relation between $r$ and $r^*$ is injective, we can write all of the relation of before in terms of $r^*$ instead of $r$, then we  write$$S_{r^*}=F_{\tau(r^*)} (\alpha_{r^*} (S_{\varphi^*(r^*)}))  $$where we also have  that  $\tau(r^*) \to 0$ and $||\varphi^*(r^*) ||_{\mathcal{C}^5}\to 0 $ for $r^* \to 0$. 
     
     As the surfaces $ S_{r^*}$ are area constraint critical points of the Hawking functional, we have that on the manifold $(\mathbb{B}_{2r_p}, g_{\tau,r},$ $k_{\tau, r} )$ they satisfy $\Phi(r^*,\tau(r^*),  \varphi^*,\lambda(r^*) )=0$ for some constants $\lambda(r^*) $. We have that $\varphi^* = \mathcal{O}(r^*)$ and then $\frac{\varphi^*}{r^*}  $ is bounded. Then  as in (\ref{expan}), we have 
     \begin{equation}\label{1b}
    \begin{split}
      -\Delta^{\mathbb{S}^2} (- \Delta^{\mathbb{S}^2} -2)\varphi^*     &= -W_1(r^*, \tau , 0, \lambda) -W_2(r^*, \tau , 0, \lambda)   \\
        &- r^{*2} \int_0^1 \int_0^1 t \Phi_{\varphi \varphi}(sr^*, \tau,st  \varphi^*,\lambda )  \frac{\varphi^*}{r^*} \frac{\varphi^*}{r^*} ds dt\\
        &- r^{*3}\int_0^1 \int_0^1 \int_0^1 s \Phi_{\varphi r r}(usr^*, \tau,ust  \varphi^*,\lambda )   \frac{\varphi^*}{r^*} du ds dt\\
        &- r^{*3} \int_0^1 \int_0^1 \int_0^1 s t \Phi_{\varphi \varphi r} (usr^*, \tau,ust  \varphi^*,\lambda )  \frac{\varphi^*}{r^*} \frac{\varphi^*}{r^*} du ds dt   + \mathcal{O}(r^{*2}) \\
        &=:r^{*2} f(r^*)
    \end{split}
\end{equation}
 where $f(r^*)$ is bounded. Then $\varphi^*$ is a solution of the elliptic PDE $ -\Delta^{\mathbb{S}^2} (- \Delta^{\mathbb{S}^2} -2) \varphi =r^{* 2} f(r^*)  $ in $K^\bot$ then, by using Schauder estimates and the injectivity of $L$ in $K^\bot$ we have $||\varphi^* ||_{\mathcal{C}^{2,\frac{1}{2}}} \leq C r^{*2}$ (for details of these result, see \cite[Chapter 6]{GilTru}). Now considering the projection to $K_0$ like in (\ref{k0}) and dividing by $r^{*2}$ we have 
 \begin{equation}
    \begin{split}
        0=&\Tilde{\pi}_0\left( \frac{\Phi(r^*,\tau,  \varphi^*,\lambda )}{r^{*2}} \right) =8 \pi (  \lambda(0) +  \frac{1}{3}\mathrm{Sc}^\tau + \frac{1}{15} |k^\tau|^2 + \frac{1}{5} (\tr k^\tau)^2) \\
         & +\Tilde{\pi} \Big( \int_0^1 \int_0^1 t \Phi_{\varphi \varphi}(sr^*, \tau,st  \varphi^*, \lambda )  \frac{\varphi^*}{r^*} \frac{\varphi^*}{r^*} ds dt\\
        &+ r^* \int_0^1 \int_0^1 \int_0^1 s \Phi_{\varphi r r}(usr^*, \tau,ust  \varphi^*, \lambda )   \frac{\varphi^*}{r^*} du ds dt\\
        &+ r^* \int_0^1 \int_0^1 \int_0^1 s t \Phi_{\varphi \varphi r} (usr^*, \tau,ust  \varphi^*,\lambda )  \frac{\varphi^*}{r^*} \frac{\varphi^*}{r^*} du ds dt \Big)  + \mathcal{O}(r^{*2})  . \\
    \end{split}
\end{equation}
Then as $||\frac{\varphi^*}{r^*} ||_{\mathcal{C}^{2}} \to 0$ for $r^* \to 0$, we have that
$$\lambda(0) =  - \frac{1}{3}\mathrm{Sc} - \frac{2}{3} |k|^2 +\frac{2}{3} (\tr k)^2.$$
Finally,  as $ 0=\pi^\bot\left( \Phi(r^*,\tau,  \varphi^*,\lambda ) \right) $ and setting $\varphi(r^*):=r^{-2} \varphi^*(r^*)$ when considering the projection to $K^\bot$ just like in (\ref{kort}),  we see that $\varphi (0)$  is given by the solution of the equation (\ref{phi_0}), then by the uniqueness of the implicit function theorems used in Theorem \ref{primfoli} the functions $\varphi(r^*)$, $\tau(r^*)$ and $\lambda(r^*)$ must agree with the ones found in the  theorem on a neighborhood of $r^*=0$. 

For $(ii)$, note that we did not use the foliation property in the previous arguments.
\end{proof}
From the proof of the previous theorem, we can also obtain directly the nonexistence result found in  \cite[Theorem 1.2]{Alex}.  Note  that for our proof, we use estimates found in \cite{Alex}.
\begin{theorem}\label{nonexis}
There exist an $\epsilon_0 >0$ such that  if at a point $p \in M$,  $\nabla(  \mathrm{Sc} + \frac{3}{5} (\tr k)^2 + \frac{1}{5} |k|^2)\neq 0 $  then there exists no concentration  of  area constrained critical spheres of the Hawking functional by surfaces satisfying $\mathcal{H}(S_r) < 4\pi +\epsilon_0^2$ and $|S_r|< \epsilon_0^2 $.
\end{theorem}
\begin{proof}
We consider $\epsilon_0$ small enough to be in the setting of the proof of the previous theorem (so small enough to apply  \cite[Proposition 3.2 ]{Alex}). Suppose  we have such surfaces and $\nabla(  \mathrm{Sc} + \frac{3}{5} (\tr k)^2 + \frac{1}{5} |k|^2)\neq 0 $.

As in the proof of the previous theorem, having that on the manifold $(\mathbb{B}_{2r_p}, g_{\tau,r},$ $k_{\tau, r} )$ our surfaces satisfy $\Phi(r^*,\tau(r^*),  \varphi^*,\lambda(r^*) )=0$ we can also consider the projection to $K_1$ and dividing by $r^{*3}$ obtain
\begin{equation}
    \begin{split}
        0=&\Tilde{\pi}_1\left( \frac{\Phi(r^*,\tau,  \varphi^*,\lambda )}{r^{*3}} \right) =\frac{4 \pi}{ 3}  \mathrm{Sc}^\tau_{,i}e_i +  \Tilde{\pi}_1 \left(  \frac{ W_2(r, \tau, 0, \lambda )}{r^{*3}}  \right) \\
         & +\Tilde{\pi} \Big( \int_0^1 \int_0^1 t \Phi_{\varphi \varphi}(sr^*, \tau,st  \varphi^*, \lambda )  \frac{\varphi^*}{r^{*2}} \frac{\varphi^*}{r^*} ds dt\\
         &+  \int_0^1 \int_0^1 \int_0^1 s \Phi_{\varphi r r}(usr^*, \tau,ust  \varphi^*, \lambda )   \frac{\varphi^*}{r^*} du ds dt\\
        &+  \int_0^1 \int_0^1 \int_0^1 s t \Phi_{\varphi \varphi r} (usr^*, \tau,ust  \varphi^*,\lambda )  \frac{\varphi^*}{r^*} \frac{\varphi^*}{r^*} du ds dt \Big).   \\
    \end{split}
\end{equation}
Then as $||\varphi^* ||_{\mathcal{C}^{2}} \leq C r^{*2}$, we find taking $r^* \to 0$ that $\frac{4 \pi}{ 3}  \mathrm{Sc}_{,i}e_i +  \Tilde{\pi}_1 \left(  \frac{ W_2(r, \tau, 0, \lambda )}{r^{*3}}  \right)_{|r=0}=0$ and proceeding as it was done for ($\ref{proj2} $) we find that $\nabla(  \mathrm{Sc} + \frac{3}{5} (\tr k)^2 + \frac{1}{5} |k|^2)= 0 $, a contradiction. 
\end{proof}

\section{Discrepancy of small sphere  limits}\label{secdis}
In this section, we will compare the small sphere limit when approaching a  point along a null cone in a spacetime $M^4 $ with  the small sphere limit  along a spacelike hypersurface $M \subset M^4$ like it was done in Section \ref{sec2}. An index $(\cdot)^4$ will denote the geometric quantities on the spacetime $M^4$. As in Section \ref{sec2}, the quantities in $M$ have no index.

Note that our critical surfaces of Theorems \ref{primfoli} and \ref{uniq} are small deformations of geodesic spheres which  satisfy that the smaller the radius, the closer the surface is to a geodesic sphere. Therefore,  to understand the discrepancy mentioned in Section \ref{secsmall}, it is a good idea to study the expansion of the Hawking energy on geodesic spheres of small radius. Recalling that the geodesic spheres are parameterized by 
 \begin{equation}
     X_G:\mathbb{R}^+ \times \mathbb{S}^2 \mapsto M,\quad (r,x) \mapsto \exp_{p}(rx)  
 \end{equation}
and that  the mean curvature of the  geodesic sphere can be expressed as 
\begin{equation}\label{H}
    H_G(x)= \frac{2}{r}-\frac{1}{3}  \mathrm{Ric}_{ij}(0) x^i x^j r- \frac{1}{4} \mathrm{Ric}_{ij;k}(0)x^ix^jx^k r^2 + \mathcal{O}(r^4).
\end{equation}
were $\mathrm{Ric}$ is evaluated at $p$. One can proceed as in \cite{Fashitam} and find that in the totally geodesic case ($k=0$), the following expansion is found
\begin{equation}
 \mathcal{E}(S_r) = \sqrt{\frac{|S_r|}{16 \pi}} \left( 1- \frac{1}{16 \pi} \int_{S_r} H^2  d\mu \right) = \frac{r^3}{12} \mathrm{Sc}_p +   \mathcal{O}(r^5)
\end{equation}
where the Hawking energy is evaluated on the geodesic sphere $S_r$ of radius $r$ and centered on a point $p$.  We can then compute, as was done in Theorem \ref{primfoli} that 
\begin{equation}
    \begin{split}
      \int_{S_r} P^2  d\mu&= 4 \pi r^2 (\tr k)^2 -2 \tr k k_{ij} \int_{\mathbb{S}^2} x^i x^j  d\mu + k_{ij}k_{pq}\int_{\mathbb{S}^2} x^i x^j x^p x^q  d\mu \\
      &= \frac{8\pi}{5} r^2  (\tr k)^2 + \frac{8\pi}{15} r^2 |k|^2
    \end{split}
\end{equation}
 with this, we then get the general expansion 
 \begin{equation}\label{geoexp}
     \mathcal{E}(S_r)=\sqrt{\frac{|S_r|}{16 \pi}} \left( 1- \frac{1}{16 \pi} \int_{S_r} H^2-P^2  d\mu \right)= \frac{r^3}{12}(\mathrm{Sc}_p  + \frac{3}{5} (\tr k)^2 + \frac{1}{5} |K|^2 ) +\mathcal{O}(r^5)
 \end{equation}
This result would agree with the result found in \cite{Alex}; therefore this gives us the idea that the problem in this discrepancy lies in the difference between the light cuts spheres and the geodesic spheres. To see this, we will follow \cite{Wang} and \cite{wangyau} in order to study in more detail the light cuts spheres and try to compare them with the geodesic spheres. 
\begin{remark}
 A natural idea would be to consider  the small sphere limit evaluating on space time constant mean curvature (STCMC) surfaces, that is, surfaces satisfying $ H^2-P^2 =4r^{-2}= \text{Constant}$. The local behaviour of these surfaces was studied in \cite{Me}, and it was shown that these surfaces are small deformations of geodesic spheres that also satisfy that the smaller the radius, the closer  the surface is to a geodesic sphere. Therefore such a small sphere limit  would also lead to (\ref{geoexp}).
\end{remark}
Let  $C_p$ be the future null cone of $p$, that is the null hypersurface generated by future null geodesics starting at $p$. Pick any future directed timelike unit vector $e_0$ at $p$, then to parameterize the light cuts $\Sigma_l$ of $C_p$ we will consider  the  map
\begin{equation}
     X_{lc}:[0, \delta) \times \mathbb{S}^2 \mapsto M^4  
 \end{equation}
such that for each point  $x \in \mathbb{S}^2  $ and $l\in [0, \delta)  $, $X_{lc}(x, l)$ is a null geodesic parameterized by the affine parameter $l$, with $X_{lc}(x, 0)=p$ and $\frac{\partial X_{lc}(x,0)}{\partial l} \in T_pM^4$ a null vector which satisfies $ \langle \frac{\partial  X_{lc}(x,0)}{\partial l} , e_0 \rangle =-1$. We define $L=\frac{\partial  X_{lc}}{\partial l} $ to be the null generator with $\nabla^4_L L=0$. We also choose a local
coordinate system $\{ u_a \}_{a=1,2}$ on $\mathbb{S}^2$ such that $\partial_a =
\frac{\partial  X_{lc}}{\partial u_a} $, $a = 1, 2$ form a tangent basis to $\Sigma_l$. We define $\Bar{L}$ to be the null normal vector along $\Sigma_l$ such that $ \langle \Bar{L} , L \rangle =-1$.  With this, we can define 
$$\sigma_{ab}^+ := \langle \partial_a, \nabla^4_{\partial_b} L  \rangle \quad \quad  \sigma_{ab}^- := \langle \partial_a, \nabla^4_{\partial_b} \Bar{L}  \rangle  $$
Then we have that the null expansions of the null cone are given by the traces $\theta^+ = \tr \sigma^+ $ and $\theta^- =\tr \sigma^-$. In this setting and with the help of normal coordinates ($y^0,\, y^i$, $i=0,..,3$ with $\frac{\partial}{\partial y_0}=e_0 $), the vectors $L$ and $\Bar{L}$ can be expressed as 
$$L= e_0 + \nu + \mathcal{O}(l) \quad \quad \Bar{L}= \frac{1}{2}(e_0 - \nu) + \mathcal{O}(l)  $$
where $ \nu = x^i \frac{\partial}{\partial x^i}$ and $x \in \mathbb{S}^2 $. We will consider a situation like in  figure \ref{figure}, that is supposing that the vector $e_0$ is a normal vector to a hypersurface $M$.    Using the results obtained in \cite{Wang} we have then that the induced metric on $\Sigma_l$ is given by 
\begin{equation}\label{metriclight}
    \begin{split}
        g^{lc}_{ab} &= l^2 \eta_{ab} +\frac{1}{3} \mathrm{Rm}^4(e_0 + \nu,\partial_a ,\partial_b,e_0 + \nu) l^2 + \mathcal{O}(l^3)\\
    \end{split}
\end{equation}
where $\eta $ is the standard metric on the sphere $\mathbb{S}^2$ and $\mathrm{Rm}^4$ is evaluated at $p$, the area of $\Sigma_l$ is given by 
\begin{equation}\label{volligt}
    |\Sigma_l| =4 \pi l^2 -\frac{2 \pi}{9}l^4 (4  \mathrm{Ric}^4 (e_0,e_0) + \mathrm{Sc}^4 ) + \mathcal{O}(l^6)
\end{equation}
 Finally, by \cite[Lemma 3.3, Lemma 3.2]{Wang}, we have that the expansions are 
 \begin{equation}
    \begin{split}
        \theta^+(l) =& \frac{2}{l} - \frac{1}{3} \mathrm{Ric}^4 (e_0 + \nu,e_0 + \nu) l+ \mathcal{O}(l^3)\\
        \theta^-(l) =& -\frac{1}{l} - \big( \frac{2}{3}\mathrm{Ric}^4 (e_0 + \nu,\frac{1}{2}(e_0 - \nu))  - \mathrm{Rm}^4 (e_0 + \nu,\frac{1}{2}(e_0 - \nu),e_0 + \nu, \frac{1}{2}(e_0 - \nu))\\
         &+ \frac{1}{6} \mathrm{Ric}^4 (e_0 + \nu,e_0 + \nu)\big) l + \mathcal{O}(l^3)\\
    \end{split}
 \end{equation}
  and therefore using that the mean curvature of $\Sigma_l$ is given by $H= \frac{\theta^+}{2} -\theta^-$ we obtain 
 \begin{equation}\label{meanli}
 \begin{split}
      H_{lc}=& \frac{2}{l} + \left(  \frac{1}{3} \mathrm{Ric}^4 (e_0 ,e_0) - \frac{1}{3} \mathrm{Ric}^4 (\nu,\nu)+\mathrm{Rm}^4 (\nu,e_0, e_0 ,  \nu) \right)l + \mathcal{O}(l^3)  
 \end{split}
 \end{equation}
where everything is evaluated at $p$. Now we want to compare the light cuts with the geodesic spheres, for this we will consider two of the surfaces with the same (small) area, that is $|S_r|=|\Sigma_l|$. First we want to find the difference between the parameters $r$ and $l$. Note that the area of a geodesic sphere of radius $r$ is given by
\begin{equation}\label{meangeo}
\begin{split}
     |S_r| =&4 \pi r^2 -\frac{2 \pi}{9}r^4  \mathrm{Sc} + \mathcal{O}(r^6)\\
     =& 4 \pi r^2 -\frac{2 \pi}{9}r^4(  \mathrm{Sc}^4 +2 \mathrm{Ric}^4 (e_0 ,e_0) - (\tr k)^2 + |k|^2  )+ \mathcal{O}(r^6)
\end{split}
\end{equation}
where in the second line we used the Gauss equation $\mathrm{Sc}= \mathrm{Sc}^4 +2 \mathrm{Ric}^4 (e_0 ,e_0) - (\tr k)^2 + |k|^2 $ (for the Lorentzian setting). Now comparing (\ref{volligt}) and (\ref{meangeo}) we can obtain the following relation 
\begin{equation}\label{r-l}
\begin{split}
     r-l=& (18- (r^2 +l^2)\mathrm{Sc}^4 )^{-1} \left( \frac{r^4}{(r+l)} (|k|^2 -(\tr k)^2) + \frac{2 (r^4 -2l^4)}{(r+l)}  \mathrm{Ric}^4 (e_0 ,e_0)  + \mathcal{O}(l^5) + \mathcal{O}(r^5) \right)\\
     =& \frac{1}{18} \left( \frac{r^4}{(r+l)} (|k|^2 -(\tr k)^2) + \frac{2 (r^4 -2l^4)}{(r+l)}  \mathrm{Ric}^4 (e_0 ,e_0)  + \mathcal{O}(l^5)+\mathcal{O}(r^5) \right)
\end{split}
\end{equation}
where we consider $r$ and $l$ to be small. As our surfaces are both parameterized over $[0, \delta) \times \mathbb{S}^2$ for some $\delta>0$, we can compare its different geometric quantities as functions. First, note that in normal coordinates, the metric of the geodesic spheres can be expressed as (by using the Gauss equation)
\begin{equation}
    \begin{split}
        g^{G}_{ab} &= r^2 \eta_{ab} +\frac{1}{3} \mathrm{Rm}(\nu,\partial_a ,\partial_b, \nu) r^2 + \mathcal{O}(r^3)\\
        &= r^2 \eta_{ab} +\frac{1}{3}\big( \mathrm{Rm}^4(\nu,\partial_a ,\partial_b, \nu)-k(\nu,\nu)k(\partial_a ,\partial_b,)+ k(\nu,\partial_a)k(\nu ,\partial_b) \big) r^2 + \mathcal{O}(r^3).
    \end{split}
\end{equation}
This expansion of the metric is similar to the one for the metric of the light cut (\ref{metriclight}), where the first term is just the metric of the round sphere. However, the second terms of the expansions are different. This would suggest that the two spheres are intrinsically different, but comparing the metrics is not enough since they are coordinate dependent quantities.  We will  compare different scalars directly to see that both spheres are geometrically distinct.   First,  we are going to compare the scalar curvature of the two spheres.  By \cite[Lemma 3.6]{Wang}, we have that the scalar curvature of the light cuts is given by 
\begin{equation}
    \mathrm{Sc}_{lc} = \frac{2}{l^2} + \mathrm{Sc}^4 + \frac{8}{3}(\mathrm{Ric}^4 (e_0 ,e_0)  -\mathrm{Ric}^4 (\nu ,\nu)  ) -4\mathrm{Rm}^4(e_0 ,\nu,e_0,\nu)  + \mathcal{O}(l^2)
\end{equation}
where $\mathrm{Sc}^4$, $\mathrm{Ric}^4$ and $\mathrm{Rm}^4$ are evaluated at $p$. Now, for the case of a geodesic sphere, we have that the Gauss curvature was calculated in  \cite{locwill} and from this we obtain 
\begin{equation}
\begin{split}
      \mathrm{Sc}_{G} =& \frac{2}{r^2} - \frac{2}{3} \mathrm{Ric}(\nu,\nu)+\mathcal{O}(r)\\
      =&   \frac{2}{r^2} - \frac{2}{3}\big(\mathrm{Ric}^4(\nu,\nu) + \mathrm{Rm}^4(\nu,e_0,e_0,\nu) - \tr k \, k(\nu,\nu) + \langle k(\nu, \cdot),  k(\cdot, \nu) \rangle  \big) +\mathcal{O}(r)  
\end{split}
\end{equation}
where as always all the quantities are evaluated in the point $p$ and $ \nu = x^i \frac{\partial}{\partial y^i}$ for $x \in \mathbb{S}^2 $.

Now, as both spheres are parameterized   on  $[0, \delta) \times \mathbb{S}^2$,  we compare the two scalar curvatures as a function over $[0, \delta) \times \mathbb{S}^2$ (assuming that they are evaluated in the same point $x\in \mathbb{S}^2$)  and  use (\ref{r-l}) to obtain
\begin{equation}
\begin{split}
     \mathrm{Sc}_{G}-  \mathrm{Sc}_{lc} =&  2T(\nu,\nu)- \frac{8}{3} \mathrm{Ric}^4 (e_0 ,e_0)   
     +\frac{2}{3}    ( \tr k \, k(\nu,\nu) - \langle k(\nu, \cdot),  k(\cdot, \nu) \rangle)\\   & - \frac{14}{3}\mathrm{Rm}^4(\nu,e_0,e_0,\nu)  + \mathcal{O}(r)  + \mathcal{O}(l^2) 
\end{split}
\end{equation}
where $T= \mathrm{Ric}^4 (\nu ,\nu) -\frac{1}{2} \mathrm{Sc}^4$.  As this quantity is in general nonzero, we conclude that the spheres are intrinsically different (note that if we consider the two functions to be evaluated in two distinct points of $\mathbb{S}^2$  the quantity is also in general nonzero).

We continue with the mean curvature of the surfaces, which gives us a measure of their  extrinsic curvature. In the case of the geodesic sphere by (\ref{H}) and the Gauss equation, its mean curvature can be expressed as 
\begin{equation}\label{H2}
    H_G(x)= \frac{2}{r}-\frac{1}{3} ( \mathrm{Ric}^4(\nu,\nu) + \mathrm{Rm}^4(\nu,e_0,e_0,\nu) - \tr k \, k(\nu,\nu) + \langle k(\nu, \cdot),  k(\cdot, \nu) \rangle) r + \mathcal{O}(r^4).
\end{equation}
  Now we compare the two mean curvatures (\ref{meanli}) and (\ref{H2}) (considering that they are evaluated in the same point $x\in \mathbb{S}^2$)  using (\ref{r-l}) obtaining after some calculations 
  \begin{equation}
  \begin{split}
       H_G- H_{lc}=& -\frac{1}{3} \Big(\frac{2}{3} \mathrm{Ric}^4 (e_0 ,e_0)  +\frac{1}{6}(|k|^2 -(\tr k)^2) +4\mathrm{Rm}^4(\nu,e_0,e_0,\nu)
        + \langle k(\nu, \cdot),  k(\cdot, \nu) \rangle\\ &- \tr k \, k(\nu,\nu) \Big)r+\mathcal{O}(r^2) +\mathcal{O}(l^2).
  \end{split}
  \end{equation}
  This result is in general nonzero (as before, even if the functions are  evaluated in two different points of $\mathbb{S}^2$). Then we have that in general, the light cuts and the geodesic spheres are intrinsically and extrinsically quite different, obtaining  different values for the Hawking energy. However, it is direct to see that if we are considering a totally geodesic hypersurface ($k=0$) then both small sphere limits will agree, and if we are also in the Minkowski space ($\mathrm{Rm}^4=0$) then the two spheres would be geometrically identical. 
  
  \begin{remark}\label{excess}
  Note that when comparing the local expansion of the Hawking energy along the critical surfaces (this is  the  expansion  (\ref{geoexp}) as the surfaces tend to  converge to geodesic spheres) with the expansion along light cuts   (\ref{lightex}), which in principle captures energy in a right way  we obtain 
  \begin{equation}
       \mathcal{E}(S_r) -  \mathcal{E}(\Sigma_l)= \frac{6}{5} |\mathring{k}|^2 l^3  + \mathcal{O}(r^5) + \mathcal{O}(l^5)>0
  \end{equation}
  where we consider $|S_r|=|\Sigma_l|$ and used (\ref{r-l}) with $l$ and $r$ small, this suggests that the geodesic spheres and the critical surfaces of the Hawking functional induce an excess of energy measured by the Hawking energy. This is a result to take into account when evaluating the Hawking energy on these surfaces. 
  \end{remark}

  \begin{remark}
   Note that the study of the small sphere limit for quasi local energies is not the only place where these geometric discrepancies are relevant. They are also present when studying small causal diamonds, as  was studied in \cite{Wang2} by Wang. The edge of a causal diamond can be thought in Minkowski space as the intersection of two light cones, a spacelike geodesic sphere emerging from the center of the diamond,  or as the light cut of one of the two cones intersecting. When considering it to be a geodesic sphere, the  Einstein tensor can be obtained by comparing  the area  of the edge (so the area of the geodesic sphere) in an arbitrary spacetime with the area of the edge in Minkowski spacetime. In \cite{Wang2}, this property was studied for the three definitions of diamonds, in higher dimensions and also in the vacuum case, obtaining different results in each case (not always proportional to the Einstein tensor) which of course diverge because of the geometric differences of the  edges.  
  \end{remark}

 \paragraph*{\emph{Acknowledgements.}} We would like to thank Jan Metzger and Claudio Paganini  for the   helpful discussions about this work and also thank Jinzhao Wang for the interesting  discussions about his results \cite{Wang2} and \cite{Wang}. This research is supported by the International Max Planck Research School for Mathematical and Physical Aspects of Gravitation, Cosmology and Quantum Field Theory.

 \appendix

\section{Some results on small surfaces}

We consider surfaces $\Sigma$ in a three dimensional Riemannian manifold $(M, g)$. If $p \in M$ and $\rho < R_p$, the injectivity radius of $(M, g)$ at $p$, we
can introduce Riemannian normal coordinates on a geodesic ball of radius $\rho$ around $p$, $B_\rho(p)$. On these coordinates, the metric can be expressed as 
\begin{equation}\label{normalcor}
    g_{ij}(r x)=  (\delta_{ij} + \sigma_{ij}(x r^2))
\end{equation}
where $\delta$ demotes the euclidean metric and  $\sigma_{ij}$ satisfies $ |\sigma_{ij}(x)| |x|^{-2}+|\partial \sigma_{ij}(x)| |x|^{-1}+|\partial^2 \sigma_{ij}(x)|  \leq \sigma_0 $. Where $\sigma_0 $ is a constant depending  on the maximum of $| \mathrm{Ric} |$, $|\nabla  \mathrm{Ric} |$ and $|\nabla^2  \mathrm{Ric}  |$ in $B_\rho(p)$.

In this context we have the following results  
\begin{lemma}[{\cite[Lemma 2.1]{ToMet}}]\label{comparison}
There exists a constant
$C$  depending only on $\rho$ and $\sigma_0$  such that for all surfaces $\Sigma \subset  B_r$ with $ r < \rho$, we have
\begin{equation}
    \begin{split}
        |\nu_\Sigma - \nu_\Sigma^E | \leq C|x|^2 \quad & \quad
        |d\mu -d \mu^E| \leq  C|x|^2\\
        |\nu -d \nu^E| \leq  C|x|^2 \quad & \quad
        |B - B^E| \leq  C(|x|+ |x|^2|B|)\\
        |R-R^E|\leq  Cr^2 R  \quad & \quad|R-R^E|\leq  Cr^2 R^E
    \end{split}
\end{equation}
Where $R := \sqrt{\frac{|\Sigma|}{4 \pi}}$ is the area radius of $\Sigma$  and the super index $E$ indicates that the quantity is evaluated with respect to the euclidean metric. In particular, the areas $|\Sigma|$ and $|\Sigma|^E$ are comparable.
\end{lemma}
In the context of the previous two lemmas, we have the following result that comes from \cite[Lemma 2.7]{ToMet} and \cite[Proposition II.1.3]{Nerzthe},   and which proofs come from the fact that the Michael-Simon-Sobolev inequality can be applied to our situation.
\begin{lemma}\label{Sobol}
For any orientable surface $\Sigma \subset B_\rho(p)$ (and $\rho$ sufficiently small), there exist a constant $C$ depending on $ \sigma_0$ and $\rho$ such that for all smooth function $f$ on $\Sigma$ we have 
\begin{equation}
  \left( \int_\Sigma f^2  d \mu\right)^{\frac{1}{2}} \leq C \int_\Sigma |\nabla f| + |H f| d\mu.
\end{equation}
Furthermore, via Hölder inequality, we have that for all $p\geq 1 $, it holds 
\begin{equation}\label{Sobol2}
     \left( \int_\Sigma f^{2p} d\mu \right)^{\frac{1}{p}} \leq C  p^2 |\text{supp} f|^\frac{1}{p} \int_\Sigma |\nabla f|^2 + |H f|^2 d\mu.  
\end{equation}
We also have that there exist a constant $c_S$ such that the Sobolev inequality, 
\begin{equation}
    ||f||_{L^2(\Sigma) } \leq c_s R^{-1}  ||f||_{W^{1,1}(\Sigma) }
\end{equation}
holds for any $ f\in \mathcal{C}^1 (\Sigma)$, where $R$ is the area radius  of $\Sigma$. From this Sobolev inequality it follows that
\begin{equation}\label{Sobolev2}
    ||f||_{L^\infty(\Sigma)} \leq 2^{\frac{2(p-1)}{p-2}} c_s R^{-\frac{2}{p}} ||f||_{W^{1,p}(\Sigma) } 
\end{equation}
for $p\in (2, \infty]$ and $ f \in W^{1,p}(\Sigma)$ and where the Sobolev norm is given by $||f||_{W^{1,p}(\Sigma) } = ||f||_{L^p(\Sigma)}+ R|| \nabla f||_{L^p(\Sigma)} $
\end{lemma}
\begin{lemma}[{\cite[Lemma 2.5]{ToMet}}]\label{Secofun}
There exists $0 < \rho_0 < \rho$ and a constant $C$ depending only on $\rho$ and $\sigma_0$ such that for all surfaces $\Sigma \subset B_r$ with $r < \rho_0$, we have
$$||\mathring{B}^E||^2_{L^2(\Sigma, \delta)} < C  ||\mathring{B}||^2_{L^2(\Sigma, g)} +Cr^4 ||H ||^2_{L^2(\Sigma, g)}$$.
\end{lemma}
We state the following result of De Lellis and Müller in the way how was used in \cite{ToMet}, a scaled version.
\begin{theorem}[{\cite[Theorem 1.1]{Lemu1}, \cite[Theorem 1.2]{Lemu2}}]\label{muller} There exists a universal constant $C$ with the following properties. Assume that $\Sigma \subset \mathbb{R}^3$ is a surface with $||\mathring{B}^E||^2_{L^2(\Sigma, \delta)}< 8 \pi$. Let
$R^E := \sqrt{\frac{|\Sigma|^E}{4 \pi}}$ be the Euclidean area radius of $\Sigma$ and $a^E := |\Sigma|_E^{-1} \int_\Sigma x d \mu^E$
be the Euclidean center of gravity. Then there exists a conformal map $\phi: S:= S_{R^E}(a^E)\rightarrow \Sigma \subset \mathbb{R}^3$ with the following properties. Let $\gamma^S$
be
the standard metric on $S$, $N$ the Euclidean normal vector field and $h$ the
conformal factor, that is $\phi^* \delta_{|\Sigma} = h^2 \gamma^S$. Then the following estimates hold
\begin{equation}
 \begin{split}
    ||H^E -2/R^E||_{L^2(\Sigma, \delta)} &\leq C ||\mathring{B}^E||^2_{L(\Sigma, \delta)} \\
    ||\phi -(a^E + id_S) ||_{L^\infty(S)} &\leq  C \,R^E ||\mathring{B}^E||_{L^2(\Sigma, \delta)},\\
    ||h^2-1 ||_{L^\infty(S)} &\leq  C \,R^E ||\mathring{B}^E||_{L^2(\Sigma, \delta)}\\
    ||\nu^E \circ \phi  -N ||_{L^2(S,\delta)} &\leq  C \,R^E ||\mathring{B}^E||_{L^2(\Sigma, \delta)}
 \end{split}   
\end{equation}
\end{theorem}
Finally, we state \cite[Lemma 3.1 and Lemma 3.2]{Janme} in our context.
\begin{lemma}\label{refine1}
Let $\Sigma \subset M$ be a surface with extrinsic diameter $d$ such that $2d$ is smaller than the injectivity radius of $M$. Then there exists a point $p_0 \in M$ with $\text{diam} (p_0, \Sigma) \leq d$ and such that in normal coordinates $\psi$ centered at $p_0$ we have that
\begin{equation}
\begin{split}
    a=\frac{1}{|\Sigma|} \int_{\psi(\Sigma)} y d\mu =0\quad \text{and} \quad
     |a_E|_E=\left|\frac{1}{|\Sigma|_E} \int_{\psi(\Sigma)} y d\mu_E \right|_E \leq C d^3
    \end{split}
\end{equation}
where $y$ denotes the position vector on $\psi(\Sigma) $. 
\end{lemma}
\begin{lemma}\label{refine2}
There exist constants $C$ and $a_0 \in (0, \infty)$ such that for every closed smooth surface $\Sigma \subset M$ with $|\Sigma| \leq a_0$ and $||\mathring{B}||^2_{L^2(\Sigma)} \leq a_0$, there exist a point $p_0 \in M$, normal coordinates $\psi : B_\rho (p_0) \rightarrow B_\rho (0)\subset \mathbb{R}^3$ and in these coordinates we have that 
\begin{equation}
    ||\frac{y}{R} -\nu ||_{L^2(\Sigma)} \leq C(R^3 +R ||\mathring{B}||_{L^2(\Sigma)} )
\end{equation}
and 
\begin{equation}
    ||\text{dist}(p_0, \cdot ) -R ||_{L^\infty(\Sigma)} \leq C(R^3 +R ||\mathring{B}||_{L^2(\Sigma)} )
\end{equation}
where $R$ denotes the area radius of $\Sigma$.
\end{lemma}
Finally, we state the following useful integrals. 
\begin{lemma}\label{integrals}
The components of a point in the sphere $\mathbb{S}^n$ satisfy
\begin{equation*}
        \begin{split}
            \int_{\mathbb{S}^n} x^i x^j d\mu =  \frac{|\mathbb{S}^n|}{n+1} \delta_{ij},
        \end{split}
    \end{equation*}
    \begin{equation*}
        \begin{split}
            \int_{\mathbb{S}^n} x^i x^j x^k x^l d\mu =  \frac{|\mathbb{S}^n|}{(n+1)(n+3)} (\delta_{ij} \delta_{kl} + \delta_{ik} \delta_{jl} + \delta_{il} \delta_{jk}),
        \end{split}
    \end{equation*}
and 
\begin{equation*}
        \begin{split}
            \int_{\mathbb{S}^n} x^i x^j x^k x^lx^p x^q d\mu =  \frac{|\mathbb{S}^n|}{(n+1)(n+3)(n+5)}&  ( \delta_{ij} \delta_{kl} \delta_{pq} + \delta_{ij} \delta_{kp} \delta_{lq}+ \delta_{ij} \delta_{kq} \delta_{lp}\\
&+ \delta_{ik} \delta_{jl} \delta_{pq}+ \delta_{ik} \delta_{jp} \delta_{lq} + \delta_{ik} \delta_{jq} \delta_{lp}\\
&+ \delta_{il} \delta_{jk} \delta_{pq} + \delta_{il} \delta_{jp} \delta_{kq} + \delta_{il} \delta_{jq} \delta_{kp}\\
&+ \delta_{ip} \delta_{jk} \delta_{lq} + \delta_{ip }\delta_{jl} \delta_{kq} + \delta_{ip} \delta_{jq} \delta_{kl}\\
&+ \delta_{iq} \delta_{jk} \delta_{lp} + \delta_{iq} \delta_{jl} \delta_{kp} + \delta_{iq} \delta_{jp} \delta_{kl} )\\
        \end{split}
    \end{equation*}
\end{lemma}

\bibliographystyle{amsplain}
\bibliography{Lit_new}
\end{document}